\documentclass[reqno]{amsart}
\usepackage{amsaddr}
\usepackage[a4paper]{geometry}

\usepackage{etoolbox}

\makeatletter
\patchcmd{\@makecaption}
{\parbox}
{\advance\@tempdima-\fontdimen2} 
{}{}
\makeatother 

\usepackage[latin1]{inputenc} 
\usepackage[T1]{fontenc}
\usepackage{verbatim}
\usepackage{amssymb,amsmath,amsfonts}
\usepackage{graphics}
\usepackage{natbib}
\usepackage{paralist}
\usepackage{pstool}
\usepackage{booktabs}
\usepackage{tikz}
\usepackage{mathtools}
\usepackage{bbm}
\usepackage{url}

\newtheorem{Assumption}{Assumption}
\newtheorem{Proposition}{Proposition}
\newtheorem{Lemma}{Lemma}
\newtheorem{Theorem}{Theorem}
\newtheorem{Corollary}{Corollary}
\newtheorem{Remark}{Remark}
\newtheorem{Definition}{Definition}
\newtheorem{example}{Example}

\DeclareMathOperator{\R}{\mathbb{R}}

\DeclareMathOperator{\pd}{\partial}

\DeclareMathOperator{\Cov}{Cov}

\def\qed{\relax\ifmmode\hskip2em \Box\else\unskip\nobreak\hskip1em $\Box$\fi}
\newcommand{\proper}{\mathsf}
\newcommand{\pP}{\proper{P}}
\newcommand{\pE}{\proper{E}}
\newcommand{\pV}{\proper{V}}

\newcommand{\pN}{\proper{N}}
\newcommand{\mv}[1]{{\boldsymbol{\mathrm{#1}}}}

\newcommand{\E}{{\bf \mathsf{E}}}
\newcommand{\V}{{\bf \mathsf{V}}}

\renewcommand{\P}{\mathbb{P}}
\newcommand{\Q}{\mathbb{Q}}

\newcommand{\Sker}{S^{\text{ker}}}
\newcommand{\Ssker}{S^{\text{sta}}}
\newcommand{\Strans}{S^{\text{trans}}}
\DeclareMathOperator{\logs}{LS}
\DeclareMathOperator{\crps}{CRPS}
\DeclareMathOperator{\scrps}{SCRPS}

\usepackage[normalem]{ulem}

\definecolor{darkgreen}{rgb}{0.2, 0.7, 0.2}

\newcommand{\changed}[1]{#1}

\graphicspath{{./figs/}}

\allowdisplaybreaks

\title[Scale invariance and robustness of scoring rules]{Local scale invariance and robustness of proper scoring rules}
\author{David Bolin$^A$ and Jonas Wallin$^B$}
\email{{\normalfont{(D.~Bolin, corresponding author)}} david.bolin@kaust.edu.sa}
\address[D.~Bolin and J.~Wallin]{
	$^A$CEMSE Division, King Abdullah University of Science and Technology, Saudi Arabia $^B$Department of Statistics, Lund University, Sweden}
\begin{document}

\maketitle

	\begin{abstract}
	Averages of proper scoring rules are often used to rank probabilistic forecasts. 
In many cases, the individual terms in these averages are based on observations and forecasts from different distributions.
		We show that some of the most popular proper scoring rules, such as the continuous ranked probability score (CRPS), give more importance to observations with large uncertainty which can lead to unintuitive rankings. 
	To describe this issue, we define the concept of local scale invariance for scoring rules.  A new class of generalized proper kernel scoring rules is derived and as a member of this class we propose the scaled CRPS (SCRPS). This new proper scoring rule is locally scale invariant and therefore works in the case of varying uncertainty. Like CRPS it is computationally available for output from ensemble forecasts, and does not require the ability to evaluate densities of forecasts.
	
	We further define robustness of scoring rules, show why this also is an important concept for average scores, and derive new proper scoring rules that are robust against outliers. The theoretical findings are illustrated in three different applications from spatial statistics, stochastic volatility models, and regression for count data. 
\end{abstract}


\section{Introduction}
A popular way of assessing the goodness-of-fit of statistical models is to quantify their predictive performance. 
This is often done by evaluating the accuracy of point predictions, using for example mean-squared errors between the predictions and the observed data. However, one is typically also interested in the ability to correctly quantify the prediction uncertainty. 
To fully quantify the prediction uncertainty one must use the entire predictive distribution, which often is referred to as probabilistic forecasting. The main method for summarizing the accuracy of probabilistic forecasts is to use averages of proper scoring rules.

Informally, a scoring rule $S(\cdot,\cdot)$ is a bivariate function that for a probability measure $\mathbb{P}$, representing a forecast, and an observed outcome $y$ returns a real number $S(\mathbb{P},y)$.
We use $S(\P,\Q)$ to denote the expected value of $S(\P,y)$ when $y \sim \Q$. The scoring rule is said to be proper if $S(\Q,\Q) \geq S(\P,\Q)$, and strictly proper if equality holds if and only if $\P=\Q$ \citep{gneiting2007strictly}. Using a strictly proper scoring rule for forecast ranking has the desirable property that, in the long run,  the best forecast is always the true distribution $\Q$. This can be seen as keeping the forecaster earnest, in the sense that he or she should always use the estimate of the probability distribution that is being predicted to get the best expected score.

Probabilistic forecasting, and the evaluation of such forecasts through proper scoring rules, is 
	used in a wide range of applications, from climate and weather prediction \citep{Palmer2002,brocker2012evaluating,techECMWF} to finance and macroeconomics \citep{garratt2003forecast,opschoor2017combining,nolde2017elicitability}. Scoring rules are routinely used in spatio-temporal statistics for model comparisons \citep{Heaton2019}, and the development of new scoring rule methods tailored for specific applications, such as extreme value statistics \citep{Lerch2017, taillardat2019extreme}, is an active research area. Over the past thirty years, the  theory of forecast evaluation using proper scoring rules has also advanced significantly, see \citet{gneiting2007probabilistic} for a review, and see \citet{gneiting2011comparing, Parry2012, dawid2014theory,Dawid2016} for more recent advancements in the field.

Throughout this work, we will mostly focus on the case of univariate and real-valued forecasts, where three popular proper scoring rules are the log-score, the Hyv{\"a}rinen score,  and the continuous ranked probability score (CRPS). Variants of these have been used in several different fields of research and applications, such as electricity price forecasting \citep{nowotarski2018recent}, wind speed modeling \citep{baran2016mixture,lerch2013comparison}, financial prediction \citep{opschoor2017combining}, precipitation modeling \citep{ingebrigtsen2015estimation}, and spatial statistics \citep{fuglstad2015does}. 

The log-score is defined as 
	$\logs(\P,y) = \log f(y)$, where $f(\cdot)$ denotes the density of $\P$ \citep{Good1952, Bernardo1979}, and has many desirable features \citep[see,e.g.][]{Roulston2003}. However, it has also been noted that it lacks robustness \citep{gneiting2007probabilistic}, which \citet{selten1998axiomatic} denoted as being hypersensitive. This is something that we will get back to later in this work. The Hyv{\"a}rinen score \citep{hyvarinen2005estimation} is also based on the density function of $\P$, and is defined as
	$$
	S(\P,y) =- \Delta  \log f(y) - \frac{1}{2} \left|\nabla \log f(y) \right|^2.
	$$
	This scoring rule is interesting since $S(\P,y)$ is unchanged if one multiplies $f$ by a constant. This means that it can be computed without knowing the normalizing constant of $f$, which 
	\citet{Parry2012} denotes as being ``homogeneous in the density function''.

The CRPS is defined as
\begin{align}\label{eq:crps_def}
	\crps(\P,y) &= - \int (F(x) - 1(x\geq y))^2 dx = \frac1{2}\pE_{\P,\P}[|X-Y|] - \pE_{\P}[|X-y|],
\end{align}
where $F(\cdot)$ is the cumulative distribution function of $\P$ \citep{gneiting2007strictly}, and $\E_{\P}[g(X)]$ denotes the expected value of a function $g(X)$ of a random variable $X\sim\P$. Furthermore, $\E_{\P,\Q}[g(X,Y)]$ denotes the expected value of $g(X,Y)$ when $X\sim\P$ and $Y\sim\Q$ are independent. 
It can be noted that the CRPS does not require the ability to evaluate the density of the forecast and thus is ideal for ensemble forecasts. Therefore, it is a main tool for verifying weather forecasts for continuous data \citep{brocker2012evaluating,hagelin2017met,descamps2015pearp,techECMWF}.

Two important concepts in probabilistic forecast evaluation is sharpness and calibration.  Calibration refers to how well the forecast $\P$ and the true target $\Q$ agrees, while sharpness is a property of $\P$ only and describes how concentrated the prediction measure is.  Whereas many comparative studies of the predictive performance of probabilistic forecasts have focused on calibration \citep[see, e.g.,][]{Moyeed2002}, \cite{gneiting2007probabilistic} argued that there should be a greater focus on assessing also the sharpness, in particular when the goal is to rank probabilistic forecasts. They proposed the paradigm of  ``maximizing the sharpness of the predictive distribution subject to calibration''  in order to evaluate predictive performance. This means that we want a forecast that has as small variability as possible while being calibrated. An important property of proper scoring rules is that they can simultaneously address calibration and sharpness. For example, the CRPS can be written as $\crps(\P,y) = \int \text{BS}(\P,y) dy$, where $\text{BS}(\P,y) = - (F(x) - 1(x\geq y))^2$ is the Brier score \citep{brier1950verification}, which can be decomposed into a calibration component and a refinement component \citep{murphy1972scalar, degroot1983comparison}. Therefore, also the CRPS can be decomposed similarly \citep{candille2005evaluation}.

The predictive performance of a model for a set of observations $\{y_i\}_{i=1}^n$ is typically assessed using an average, aka.\ composite \citep{dawid2014theory}, score 
\begin{equation}\label{eq:average_score}
	S_n = \frac1{n}\sum_{i=1}^n S(\P_{i},y_i),
\end{equation}
where $\P_{i}$ denotes the predictive distribution for $y_i$ based on the model. These different predictive distributions might, for example, be one-step-ahead predictions in a time series model or leave-one-out cross-validation predictions for a model in spatial statistics.  If $S$ is a proper scoring rule then the average score is also proper.  Therefore, the use of average scores for model comparison is natural, but we will show that it may lead to unintuitive forecast rankings if popular scoring rules such as the CRPS or the Hyv{\"a}rinen score are used. There are two main reasons for this: Firstly, there is always some degree of model misspecification for the models that are compared, and one model is rarely  best for all observations. Secondly, the predictive distributions  typically vary between different observations, and depending on the scoring rule, each observation may therefore not be equally important for the average score. 

A situation where the latter problem occurs is when the observations have varying degrees of predictability, which often causes the variances of the predictive distributions to be different.
This, for example, occurs frequently for weather and climate data which have large spatial and temporal variability in their predictability \citep{campbell2005}. In this case, when evaluating the score of model through the average score \eqref{eq:average_score} one does not evaluate a single predictive distribution, $\P$, against a single truth, $\Q$,  but rather a set of model-generated predictions $\{\P_i\}$ against  a set of truths $\{\Q_i\}$. This is an overlooked factor in many applications of forecast evaluation through the use of proper scoring rules. To illustrate this, we consider the following simple example in the spirit of those in \cite{gneiting2007probabilistic}.

	\begin{example}
		\label{ex:toy}
		Suppose that nature generates data from the  model $Y_i | \sigma^2_i \sim  \mathcal{N}\left(0, \sigma^2_i \right)$, where $\sigma^2_i \sim   \pi(\sigma^2)= \frac{7}{8}\delta_{\frac{8}{14}}(\sigma^2) + \frac{1}{8}\delta_{\frac{8}{2}}(\sigma^2)$.
		There are four different forecasters: The ideal forecaster $\P^1_i = \mathcal{N}\left(0,\sigma^2_i  \right)$, the climatological forecaster (which uses the marginal distribution)   $ \P^2_i =  \int  \mathcal{N}\left(0, \sigma^2 \right) $ $\pi(\sigma^2) d\sigma^2$, the confident $\P^3_i = \mathcal{N}\left(0,\frac{8}{14} \right) $ and the pessimistic $\P^4_i = \mathcal{N}\left(0, \frac{8}{2}\right)$.  If a proper scoring rule is used, the ideal forecaster should be best in the long run; however, the ranking of the other three will depend on which scoring rule that is used.
		We generate $2000$ observations and rank the forecasters using the three previously mentioned scoring rules through average scores. The ranking is shown in Table \ref{tab:rank}. Note that the CRPS ranks the confident forecaster as the second best, whereas the Hyvärinen score ranks this as the worst forecaster. As the log score, the  Hyv{\"a}rinen score ranks the climatological forecaster as second but it ranks the pessimistic forecaster higher than the confident.
		\begin{table}
			\label{tab:rank}
			\caption{Rankings of the forecasts for each score in Example~\ref{ex:toy}. The value in the parentheses is the score value.}
			\begin{tabular}{ccccc}
				\hline
				Forecaster & density& CRPS & log score & Hyv{\"a}rinen score \\
				\hline
				Ideal & $\mathcal{N}\left(x;0,\sigma^2_i  \right)$ 
				&  1(-0.52)  & 1 (-1.27) & 1 (0.76)\\
				Climatological & $   \frac{7}{8}\mathcal{N}\left(x;0,\frac{8}{14}  \right) + \frac{1}{8}\mathcal{N}\left(x; 0,\frac{8}{2}  \right)$  
				&3 (-0.55) & 2 (-1.36)& 2 (0.65) \\
				Confident & $  \mathcal{N}\left(x;0, \frac{8}{14}   \right) $ 
				&2 (-0.54)  & 3 (-1.52)& 4 (0.2) \\
				Pessimistic & $  \mathcal{N}\left(x;0,\frac{8}{2}   \right) $  
				& 4 (-0.65) &4 (-1.74)& 3(0.22) \\
				\hline
			\end{tabular}
		\end{table} 
\end{example}

The preferred ranking of the forecasters will naturally depend on the context of how the forecasts are used, but the forecast evaluator should be aware of the fact that the use of a particular scoring rule implicitly will determine how important different observations are. To guide the evaluator with respect to this we will later introduce the concept of scale functions of scoring rules.

One previously proposed strategy for dealing with situations where on has varying uncertainty is to
use so-called skill scores, which typically take the form
$(S_n - S_n^{ref})/S_n^{ref}$ or $(S_n - S_n^{ref})/(S_n^{opt} - S_n^{ref})$.
Here $S_n$ is the score by the forecaster, $S_n^{ref}$ is a score for a reference method \citep{Winkler1996}, and $S_n^{opt}$ is a hypothetical optimal forecast  \citep[][Chapter 7.1.4]{wilks2005}. This standardization may seem natural since the score equals $1$ for the optimal forecast, is positive whenever the forecaster is better than the reference, and negative otherwise. It could also solve the problem with varying predictability if this information is present in the reference method. However, skill scores are in general improper even if they are based on a proper scoring rule \citep{murphy1973, gneiting2007strictly}. 

The main focus of this work is to introduce proper alternatives to skill scores. To that end, we  will define and analyze properties for average scoring rules in the case of varying predictability, by examining differently scaled observations and predictions.  %

One of our main results is to propose a method for standardizing proper scoring rules in a way so that they remain proper. More specifically, the main contribution of this work is twofold:
\begin{enumerate}
	\item[(i)] We introduce two properties of scoring rules which are important if average scores are used for forecast evaluation: Local scale invariance and robustness. We show that popular scoring rules such as the CRPS, he Hyv{\"a}rinen score, the mean square error (MSE), and mean absolute error (MAE) lack both of these properties and illustrate why this can be a problem through several examples and applications in spatial statistics, regression modeling, and finance. 
	\item[(ii)] We derive a new class of proper scoring rules that maintain the good properties of the CRPS, such as easy-to-use expressions that facilitate straight-forward Monte Carlo (or ensemble) approximations of the scores in cases when the density of the model is unknown. Among these new proper scoring rules, which can be seen as generalizations of the proper kernel scores \citep{dawid2007geometry}, we show that there are some special cases which are scale invariant and some that are robust, which solves the problems encountered in the above-mentioned applications.
\end{enumerate}
An example of a new proper scoring rule from the proposed class is the scaled CRPS (SCRPS) 
\begin{align}\label{eq:scrps}
	\scrps(\P, y) :=& - \frac{\E_{\P}\left[ |X-y| \right]}{\E_{\P,\P}\left[ |X-Y| \right]} - \frac{1}{2}\log\left( \E_{\P,\P}\left[ |X-Y| \right]\right)\\
	=& -\frac1{2}\left(1 + \frac{\crps(\P,y)}{\crps(\P,\P)}+\log(2|\crps(\P,\P)|)\right),\notag
\end{align}
where the second equality is obtained using the definition of the CRPS in \eqref{eq:crps_def}. Compared to the CRPS, the SCRPS has the desirable property that the penalty of an incorrect prediction, $\E_{\P}\left[ |X-y| \right]$, is standardized by $\E_{\P,\P}\left[ |X-Y| \right]$. This means that if the prediction expects a large uncertainty, then the penalty of making a large error is down-weighted. This is essential for being locally scale invariant, which we show that the SCRPS is. 
Since the expectations used in CRPS and SCRPS are the same, one can compute them equally easily for ensemble forecasts, and in fact, from the second equality we see that we can compute the SCRPS for any distribution where we can compute the CRPS and its expectation. Thus, the SCRPS is relatively easy to compute for many standard distributions, and as an example we provide analytic expression for Gaussian distributions in Appendix~\ref{app:gauss}.
	The structure of the article is as follows.  In Section~\ref{sec:scale} and Section~\ref{sec:robust}, we respectively introduce the concepts of local scale invariance 
	and robustness for scoring rules, and provide illustrative examples of why the lack of these properties may lead to unintuitive conclusions when ranking forecasts. In Section~\ref{sec:kernel_scores}, we analyze the class of kernel scores, which contains the CRPS as a special case, in terms of robustness and local scale invariance. In the section, we also propose a robust version of the CRPS. 
	In Section~\ref{sec:standardized_scores} a new general family of scoring rules, of which the SCRPS is a special case, is introduced and analyzed in terms of local scale invariance and robustness. Section~\ref{sec:application} presents three different applications where we illustrate the benefits of the new scoring rules. The article concludes with a discussion in Section~\ref{sec:discussion}. The article contains three appendices with (\ref{app:gauss}) formulas for the new scoring rules for Gaussian distributions; (\ref{sec:entropy}) a characterization of scoring rules in terms of generalized entropy; and (\ref{sec:proofs}) proofs of certain propositions.
	


\section{Local scale invariance of scoring rules}\label{sec:scale}
In this section, we define the concept of local scale invariance of proper scoring rules. We begin by some motivating examples, then provide the mathematical definition and some properties, and finally provide a discussion about local scale invariance and related concepts.

\subsection{Motivation}\label{sec:motivation}
As previously noted, the predictive measures $\P_{i}$ in \eqref{eq:average_score} often have varying uncertainty. Two common causes for this are non-stationary models and irregularly spaced observation locations. 
In the case of varying uncertainty, the magnitude of the value given by the scoring rule at each location can depend on the magnitude of this uncertainty. This is what we will refer to as scale dependence, or the lack of local scale invariance. 

If the average score \eqref{eq:average_score} is used to compare the models, the scale dependence will make the different observations in the sum contribute differently to the average score. In other words, the importance of the predictive performance for the different observations will depend on how the scoring rule depends on the uncertainty. This means that emphasis may be put on accurately predicting observations with certain scales, which can cause unintuitive results. 

\begin{figure}[!t]
	\begin{center}
		\begin{minipage}[t]{0.24\linewidth}
			\begin{center}
				\phantom{y}CRPS\\
				\includegraphics[width=\textwidth]{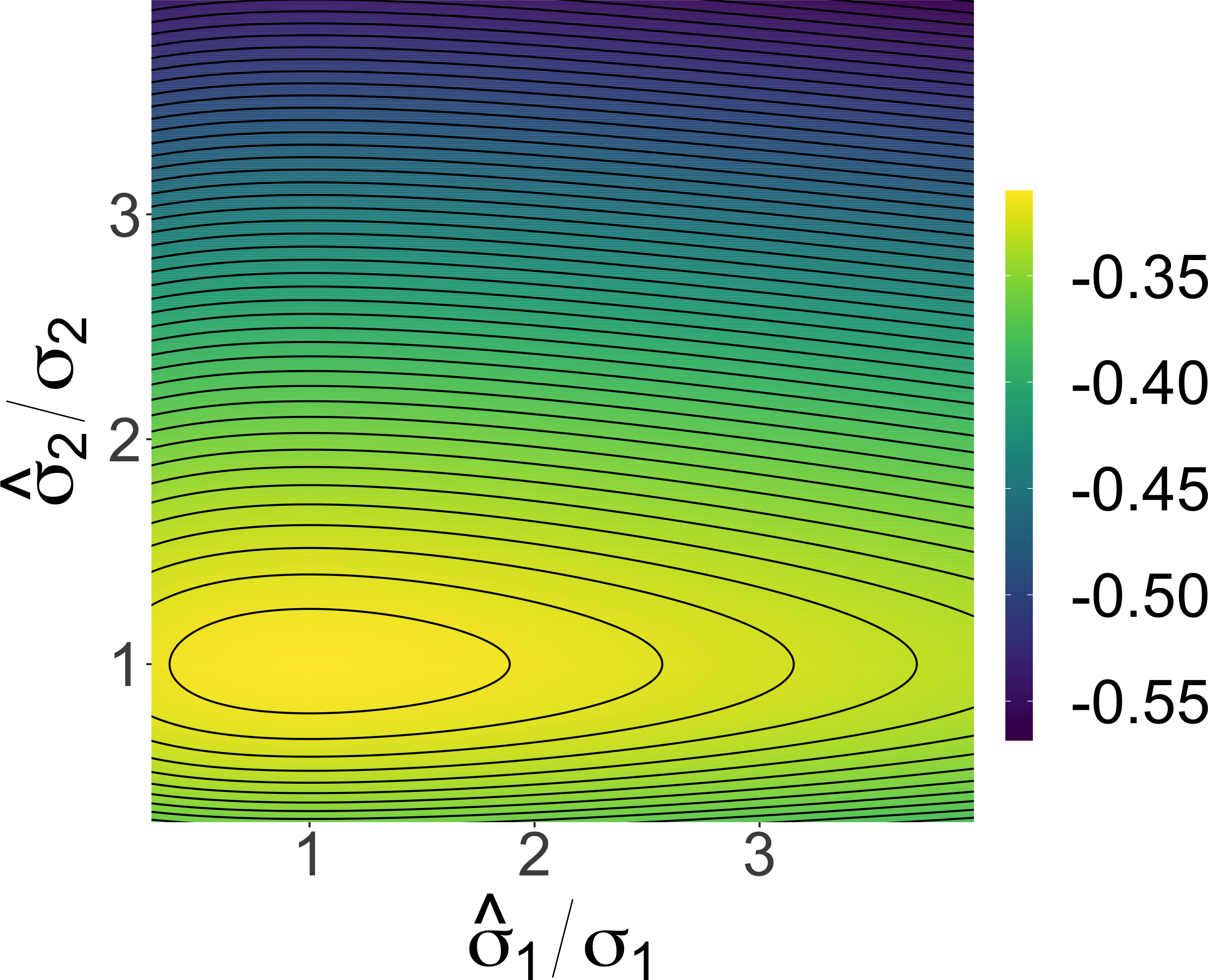}\\
				\includegraphics[width=\textwidth]{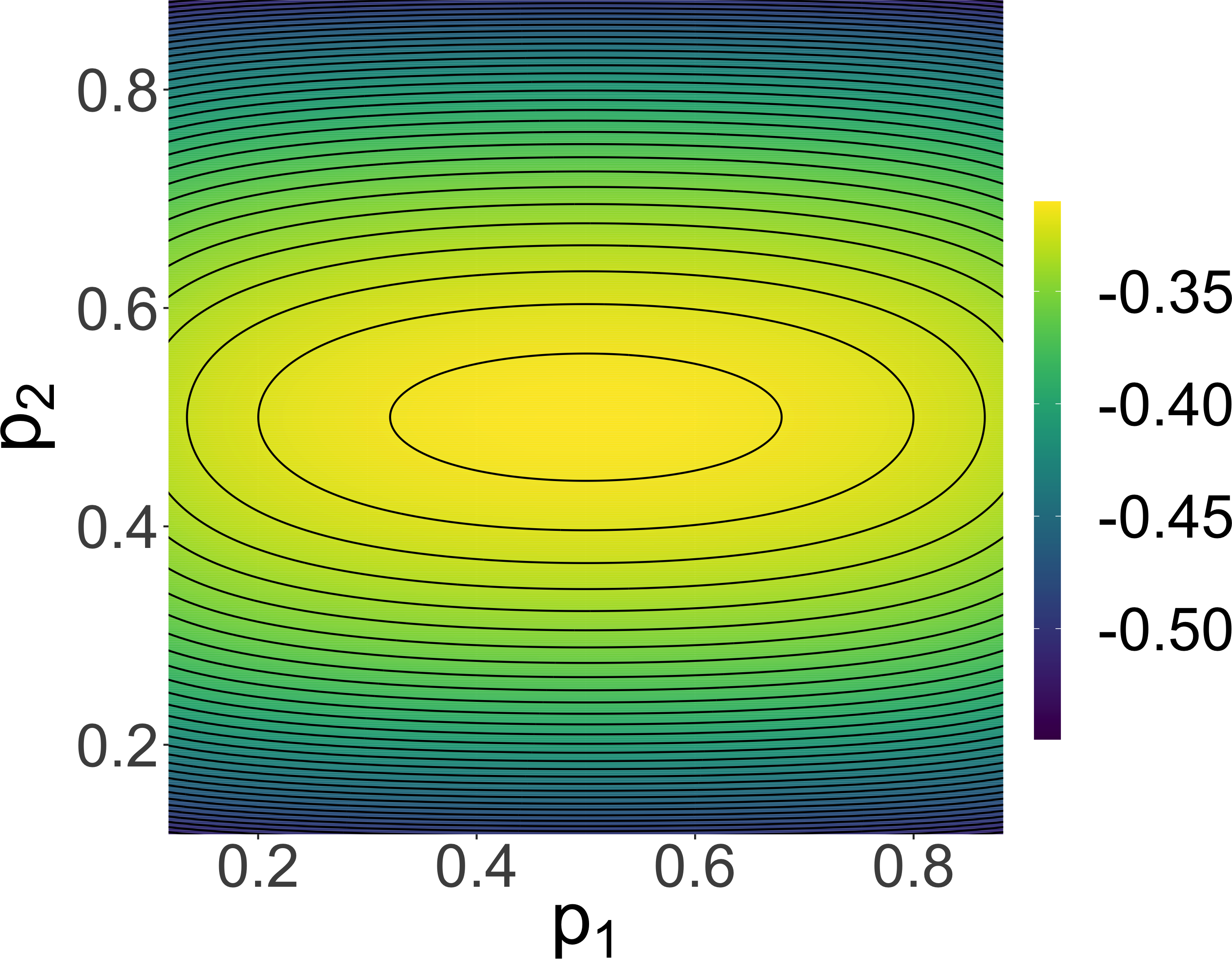}
			\end{center}
		\end{minipage}
		\begin{minipage}[t]{0.24\linewidth}
			\begin{center}
				\phantom{y}log(LS)\\
				\includegraphics[width=0.93\textwidth]{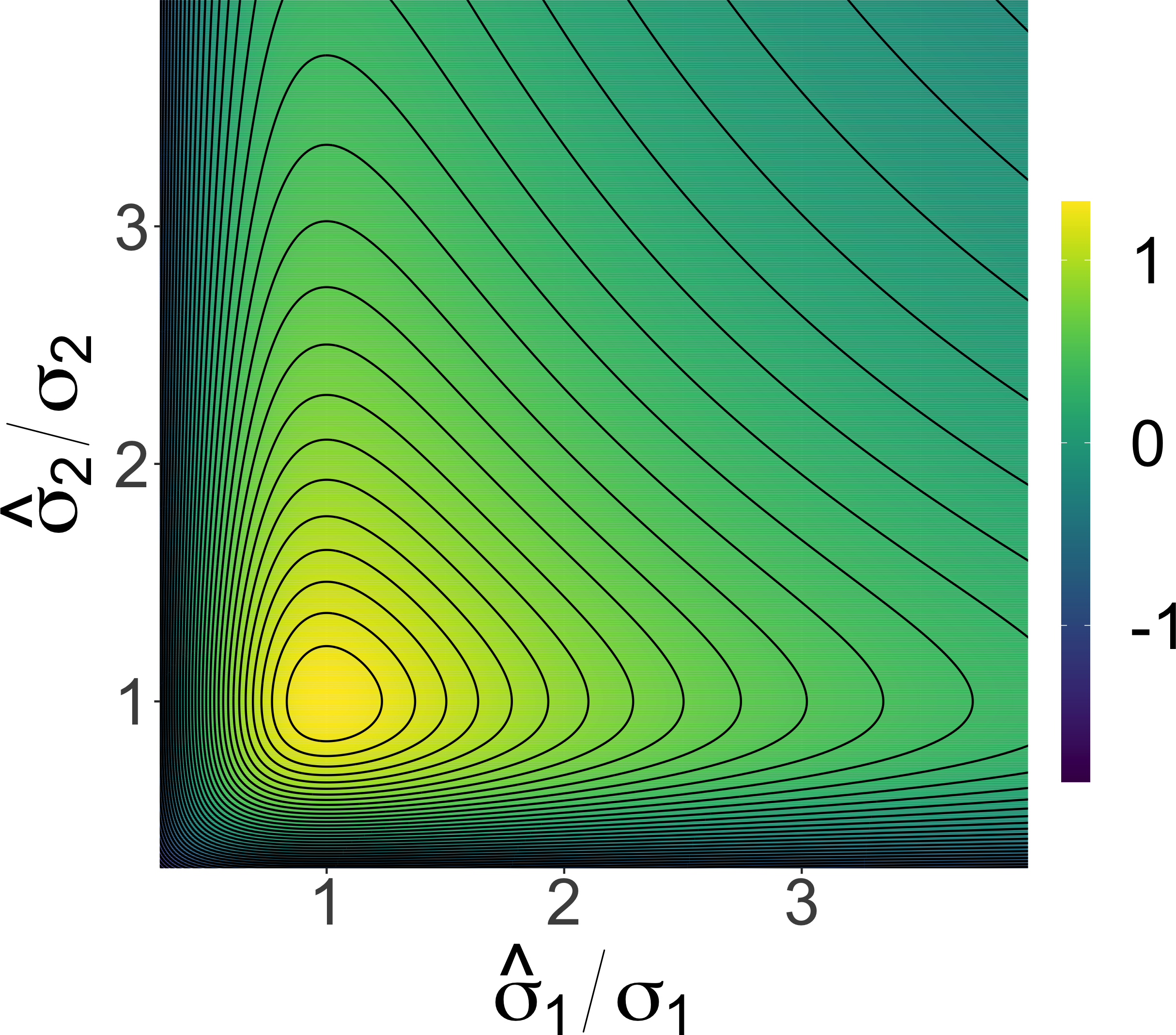}\\
				\includegraphics[width=0.96\textwidth]{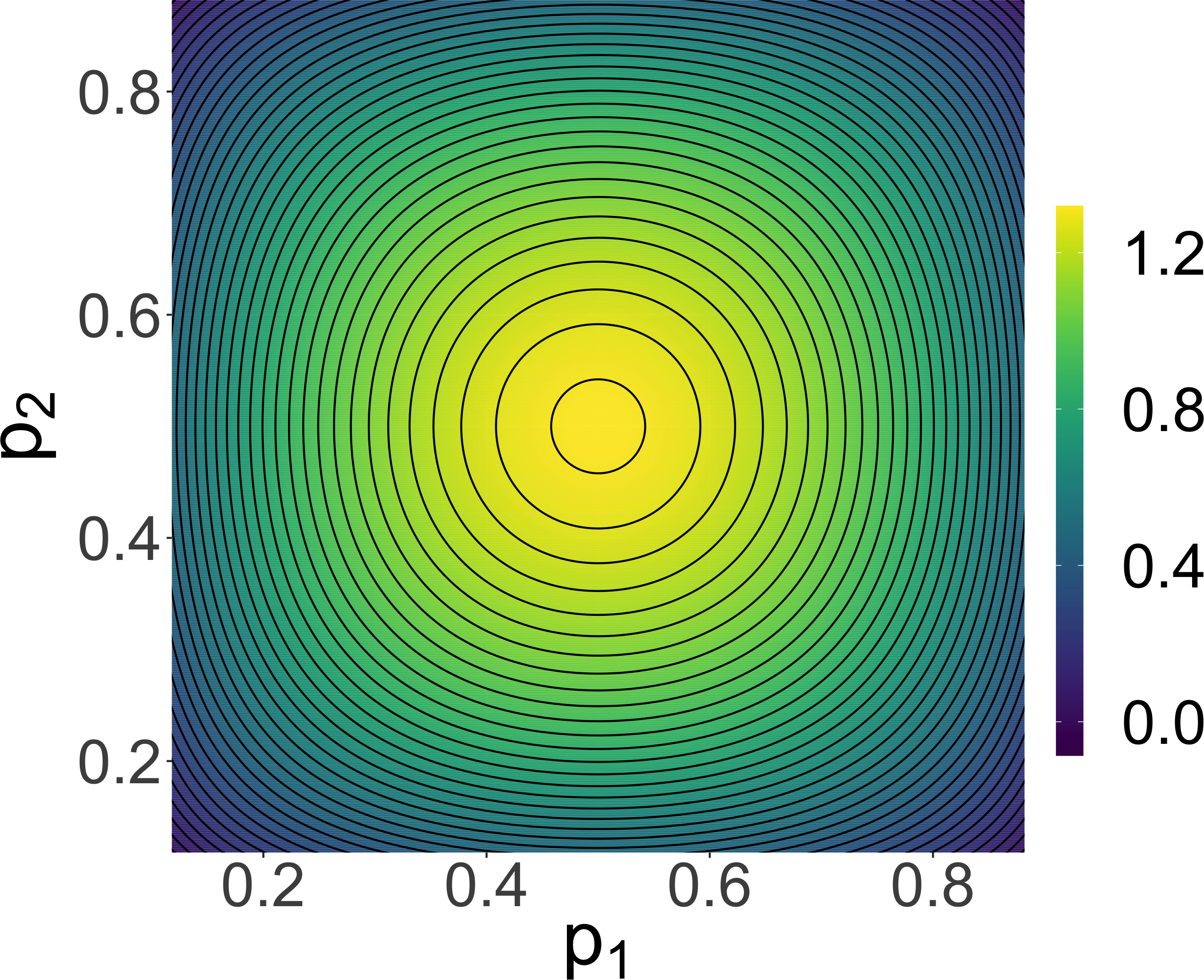}
			\end{center}
		\end{minipage}
		\hspace{0.04cm}
		\begin{minipage}[t]{0.24\linewidth}
			\begin{center}
				\phantom{y}SCRPS\\
				\includegraphics[width=\textwidth]{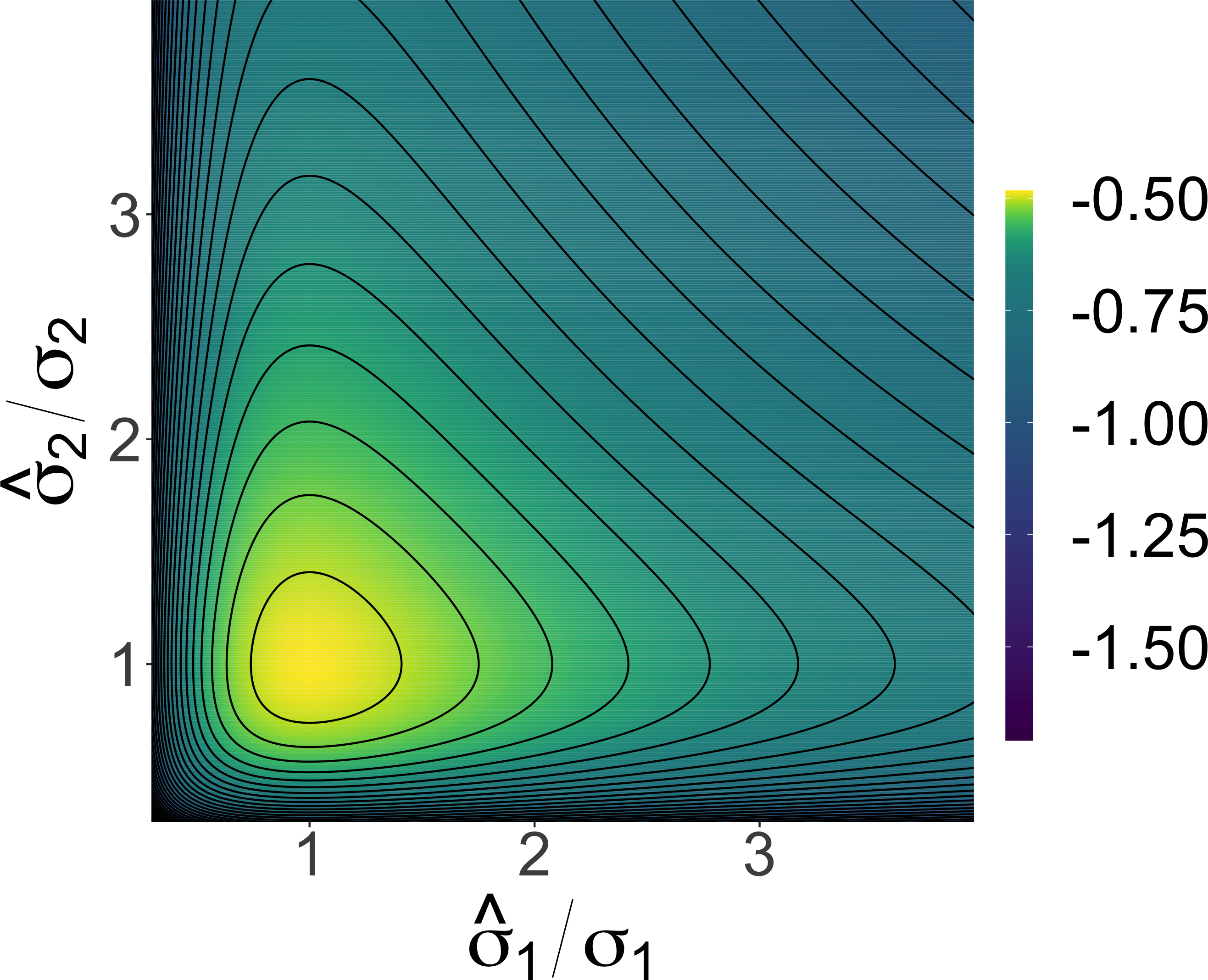}\\
				\hspace{-0.2cm}\includegraphics[width=0.975\textwidth]{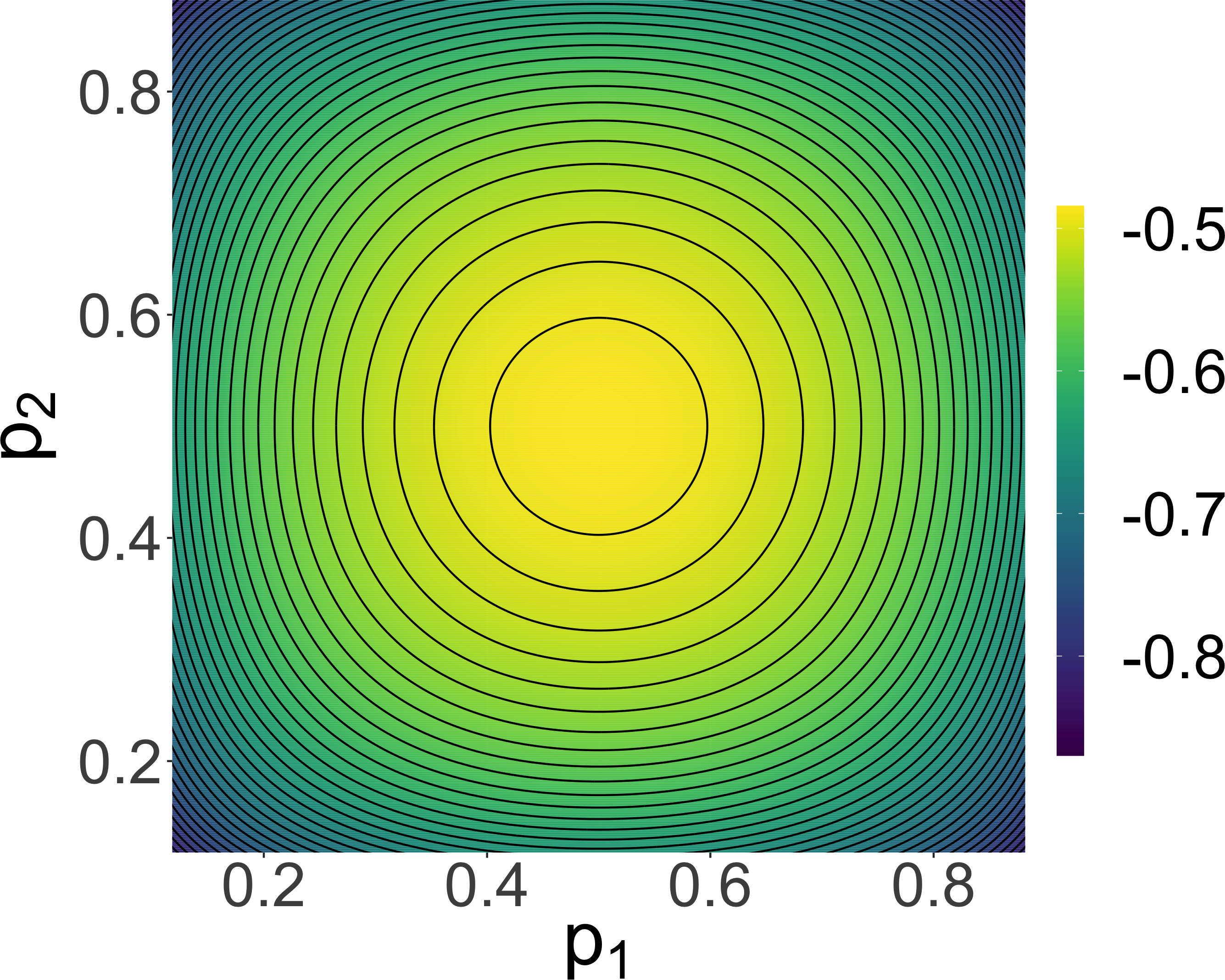}
			\end{center}
		\end{minipage}
		\begin{minipage}[t]{0.24\linewidth}
			\begin{center}
				Hyvärinen\\
				\includegraphics[width=0.96\textwidth]{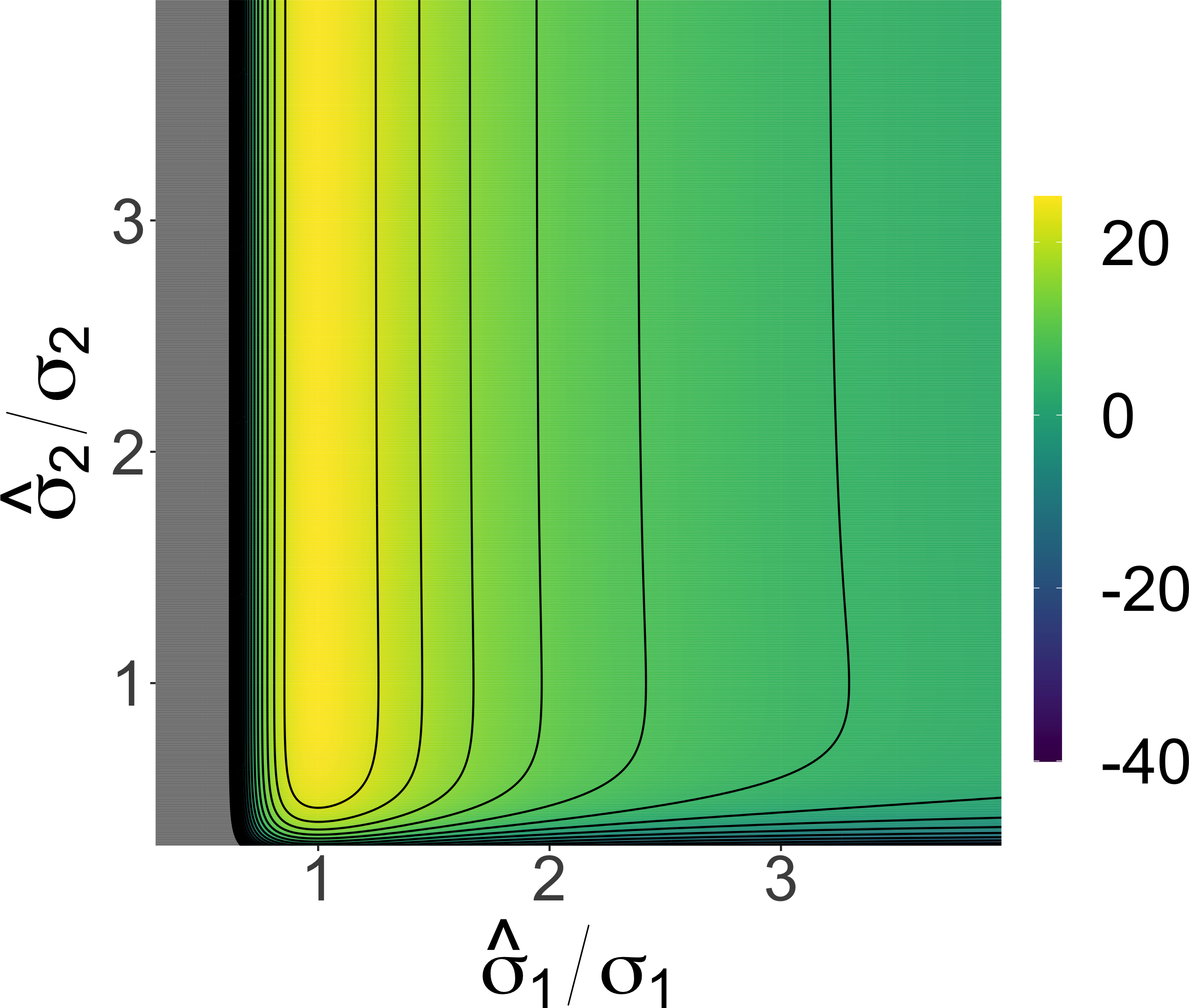}\\
				\includegraphics[width=0.96\textwidth]{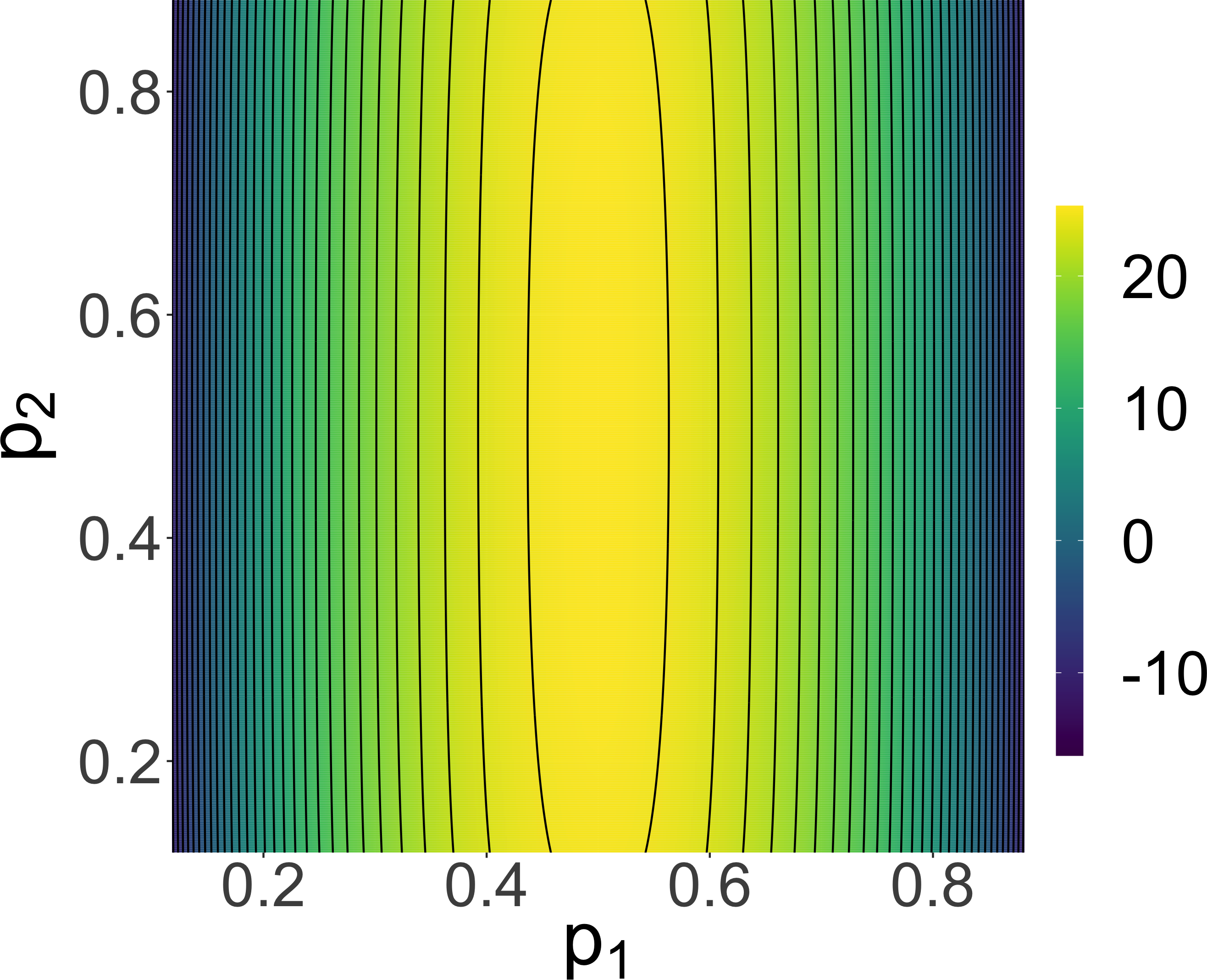}
			\end{center}
		\end{minipage}
	\end{center}
	\vspace{-0.4cm}
	\caption{Average expected CRPS, log-score, SCRPS, \changed{and Hyvärinen score} for two mean-zero normal distributions with $\sigma_1 = 0.1$ and $\sigma_2 = 1$. The top row shows the values as functions of the relative errors in the standard deviations of the predictive distributions when the predictive model has the correct mean value. The bottom row shows the values as functions of probabilities $p_1$ and $p_2$, when the predictive model has correct variances but mean values $\hat{\mu}_i = \sigma_i\Phi^{-1}(p_i)$. \changed{The color scale for the Hyvärinen score has been truncated from below.}}
	\label{fig:bivar}
\end{figure}

\begin{example}\label{ex1}
	Consider two observations $Y_i \sim \Q_{\theta_i}= \pN(0,\sigma_i^2), i=1,2,$ with $\sigma_1 = 0.1$ and $\sigma_2 = 1$. Assume that we want to evaluate a model which has predictive distributions  $\P_{i} = \pN(\hat{\mu}_i,\hat{\sigma}_i^2)$ for $Y_i$, using the average of a proper scoring rule $\frac12 S(\P_{1},Y_1) +\frac12 S(\P_{2},Y_2)$. Using the expressions in Appendix \ref{app:gauss}, we compute the expected average score when the CRPS, \changed{the log-score, the SCRPS, and Hyvärinen score} are used, and investigate how the average scores depend on the model parameters. The top row of Figure~\ref{fig:bivar} shows the results when the two standard deviations, $\hat{\sigma}_i$, are varied while keeping $\hat{\mu}_i=0$ fixed. In the bottom row, the average scores are instead shown as functions of $\hat{\mu}_i$, while $\hat{\sigma}_i = \sigma_i$. To simplify interpretation, the size of $\hat{\mu}_i$ is shown as a quantile of the true distribution, i.e., $\hat{\mu}_i = \sigma_i\Phi^{-1}(p_i)$ where $p_i$ is a probability and $\Phi$ is the CDF of the standard Gaussian distribution. In both rows, one can note that the average CRPS is much more sensitive to relative errors in the second variable, which has the higher variance. Thus, if we would compare  two competing models for this example using CRPS, the model which has the better prediction for the second variable will likely win, even if it is much worse for the first variable.  As seen in the figure, this is not the case for the log-score or for the SCRPS, \changed{whereas the Hyvärinen score is more sensitive to relative errors in the first variable.}
\end{example}
The example shows that, in the case of normal distributions, the CRPS, \changed{the Hyvärinen score}, and the log-score handle varying uncertainty (scaling) differently. For the log-score, the average score is symmetric in the two observations, whereas the \changed{Hyvärinen score punishes errors in the variable with smaller uncertainty more, and the} CRPS punishes errors in the variable with larger uncertainty more.
This has important implications for the case when the average scores are used to evaluate statistical models with dependence. For example, if different random field models for irregularly spaced observations are evaluated using leave-one-out cross validation, the predictive distributions for locations that are far away from other locations will have larger variances and thus be more important when computing the average CRPS. Hence, it will be most important to have small errors for observations without any close neighbors, rather than giving accurate predictions for locations where there is much data to base the prediction on.  \changed{We formalize these observations in the following subsection and use them to give a definition of local scale invariance.}

\subsection{Definition and properties}
\changed{In the introduction we gave an informal definition of a scoring rule, and we begin this subsection by formalizing this notion and then define the concept of local scale invariance. Let $\overline{\mathbb{R}} = [-\infty,\infty]$ denote the extended real line and $\Omega$ some sample space. Further, let $\mathcal{P}$ be a set of probability measures on $(\Omega, \mathcal{F})$, where $\mathcal{F}$ denotes a $\sigma$-algebra of subsets of $\Omega$. A scoring rule $S : \mathcal{P} \times \Omega \rightarrow \overline{\mathbb{R}}$ is a function of a probability measure $\P\in \mathcal{P}$  representing the forecast and an observed outcome $y\in\Omega$ such that $S(\P,\cdot)$ is measurable for each $\P\in\mathcal{P}$ \citep[see, e.g.,][]{gneiting2007probabilistic}.} 
To analyze the issue of scale dependence, we consider the scenario that the set of optimal predictive distributions in \eqref{eq:average_score} are location-scale transformations of some common base measure $\Q$ on $(\mathbb{R},\mathcal{B}(\mathbb{R}))$, where $\mathcal{B}(\mathbb{R})$ is the Borel $\sigma$-algebra. 
That is $Y_i\sim \Q_{\theta_i}$ where $\theta_i=(\mu_i,\sigma_i)$ contains the location and scale parameters of the observations and $\Q_{\theta_i}$ is the probability measure of the random variable $\mu_i  + \sigma_i Z$ where $Z \sim \Q$.
\changed{We are, e.g., in this scenario whenever we are performing prediction for Gaussian processes and random fields, which is a common scenario in spatio-temporal statistics.}
To investigate scale dependence it is enough to study how the scale affects the score for one of the observations, so we therefore drop the indexing by $i$ and consider a generic location-scale transformation $\Q_{\theta}$ of $\Q$.
If $\Q$ has a density \changed{$q(\cdot)$} with respect to the Lebesgue measure, then $\Q_\theta$ has density \changed{$q((\cdot-\mu)/\sigma)/\sigma$}.


\begin{Definition}\label{def:measuresF}
	Let $S$ be a scoring rule on $\Omega$ then $\mathcal{P}_{S}$ is the set of probability measures on $(\Omega, \mathcal{F})$ such that if $\mathbb{P} \in \mathcal{P}_S$ and $\mathbb{Q}\in \mathcal{P}_S$ then $|S(\mathbb{P},\mathbb{Q}) | < \infty$.
	\changed{For a set $\mathcal{Q}_0$ of probability measures, we write  $\mathcal{Q} = \{\Q_{\theta} : \Q \in \mathcal{Q}_0, \theta \in \mathbb{R}\times\mathbb{R}^+  \} $ for the set of probability measures in $\mathcal{P}_S$  which can be obtained as location-scale transforms of measures in $\mathcal{Q}_0$.}
\end{Definition}

At first sight, a natural way to remedy issues with varying scale is to require that the scoring rule is scale invariant in the sense that $S(\P_{\mu,\sigma},\Q_{\mu,\sigma})=S(\P_{0,1},\Q_{0,1})$. However, for a scoring rule to have this property  it needs to disregard the sharpness of the forecast \citep{gneiting2007probabilistic}, and would therefore not be ideal for ranking predictive distributions. 

As an alternative, we propose the concept of local scale invariance.
\changed{In order to define this property we examine the geometry defined by the scoring rule. If $S(\Q_{\theta},\Q_{\theta'})$ is twice differentiable with respect to $\theta$ then 
	$S(\Q_{\theta+ d\theta},\Q_{\theta})$ defines a Riemannian Metric on $\mathcal{Q}$ with
	$$
	S(\Q_{\theta+ d\theta},\Q_{\theta}) = d\theta^T s(\Q_{\theta}) d\theta,
	$$
	see, e.g., \cite{dawid2014theory}. We will refer to $s(\Q_{\theta})$ as the scale function of the scoring rule $\mathcal{Q}$ and introduce the following definition of local scale invariance.
	\begin{Definition}
		\label{def:scalefunction}
		Let $S$ be a proper scoring rule with respect to some class of probability measures $\mathcal{P}$ on $(\R,\mathcal{B}(\R))$, and assume that $\mathcal{Q}_0$ is a set of probability measures such that $\mathcal{Q}\subseteq \mathcal{P}$ and $\mathcal{Q} \subset \mathcal{P}_S$. If 
		$s(\Q_{\theta})$ exists and satisfies $s(\Q_{\theta})\equiv \frac{1}{\sigma^2}s(\Q_{(0,1)})$, we say that $S$ is locally scale invariant on $\mathcal{Q}$.
	\end{Definition}	
	Consider a situation where one makes a small location-scale misspecification proportional to the scale of the true distribution. Definition~\ref{def:scalefunction} then implies that, for a locally scale invariant scoring rule, the difference in the score between the correct and misspecified models should not depend on the scale.  To see this,  let the predictive distribution be formulated as $\Q_{\theta + t\sigma \delta}$, where $\theta + t\sigma \delta = (\mu + t\sigma\delta_1, \sigma + t\sigma\delta_2)$ for $t>0$ and a direction $\delta = (\delta_1,\delta_2)$ with $\|\delta\|=1$. Considering how the scoring rule behaves as a function of the shift size, $t$, we have as $t \searrow 0$,
	$$
	S(\Q_{\theta},\Q_{\theta})-S(\Q_{\theta + t\sigma \delta},\Q_{\theta})= t^2\sigma^2 \delta^T s(\Q_{\theta})\delta +o(t^2).
	$$
	Thus, if $s(\Q_{\theta})$ is locally scale invariant, the misspecification error is independent of the scale. We will discuss this point further in Subsection \ref{sec:ForecastEval}.
	\begin{Remark}
		It can be noted that the definition of local scale invariance could be extended to models which are not location-scale families, by interpreting $\sigma$ as the standard deviation of $\Q_{\theta}$. However, in that case, where $\theta$ is a general parameter, using the shift $\theta + t\sigma \delta$ is less natural. There are various approaches that could be taken to overcome this issue, but for brevity we leave this generality for future work and focus on  location-scale families here. 
	\end{Remark}
}

In order to ensure the existence of the scale function, we need some assumptions on the probability measure $\Q_\theta$ in Definition \ref{def:scalefunction}.
The following assumption (where the specific value of $\alpha$ varies with the scoring rule) is sufficient for the scoring rules considered in this article. 
\begin{Assumption}\label{ass1} 
	The probability measure $\Q$ on $(\mathbb{R},\mathcal{B}(\mathbb{R}))$ has density \changed{$\exp(\Psi)$} with respect to the Lebesgue measure for some twice differentiable function \changed{$\Psi$} such that
	$\E_{\Q}[\Psi'(X)]$, $\E_{\Q}[\Psi''(X)]$, and $\E_{\Q}[\left( \Psi'(X) \right)^2]$ are finite. \changed{Further, there exists $\alpha \geq 0$ such that}  $\E_{\Q}[ |X|^{\alpha} ]$, $\E_{\Q}[ |X|^{\alpha+1} \Psi'(X)]$, $\E_{\Q}[|X|^{\alpha+3}\Psi''(X) ]$, and $\E_{\Q}[|X|^{\alpha+3} \left( \Psi'(X) \right)^2]$ are finite.
\end{Assumption}
These assumptions are satisfied for all examples with continuous distributions we will study later on.
However, it should be noted that the assumptions are far from necessary for the scoring rules considered and the corresponding scale functions thus exist for a much larger class of measures. Since our objective is  to highlight the meaning of the scale function rather than finding the largest class of measures for which it exists,  we will prove the results under these assumptions to simplify the exposition. The following result shows that the log-score is locally scale invariant. It is proved in Appendix~\ref{sec:proofs}.

\begin{Proposition}\label{prop:scale_ls}
	Let $\mathcal{Q}$ denote a set of location-scale transformations $\Q_\theta$ that satisfy Assumption \ref{ass1} with $\alpha=0$. Then the log-score has scale function 	$s(\Q_{\theta})= \frac{1}{\sigma^2}H_{\Q}$ on $\mathcal{Q}$, where $H_{\Q}$ is a $2\times 2$ matrix independent of $\theta$.
\end{Proposition}

\changed{
	We will later show that the SCRPS also is locally scale invariant, whereas the CRPS has a scale function that satisfies $s(\Q_{\theta}) =\sigma  s(\Q)$.  This difference in the scale function thus captures the behaviour seen in Figure \ref{fig:bivar}, in the sense that the CRPS is sensitive to the scale of the observations whereas the SCPRS and the log-score are not. It is also easy to see that the scale functions for the proper scoring rules defining MSE and MAE satisfy $s_{MSE}(\Q_{\theta}) =\sigma^2  s_{MSE}(\Q)$ and $s_{MAE}(\Q_{\theta}) =\sigma  s_{MAE}(\Q)$ respectively. Hence, neither is locally scale invariant.}

	\subsection{Local scale invariance and discriminatory power of composite scores}
	\label{sec:ForecastEval}
	When using a scoring rule to rank forecasts (or models) it is important to understand how it discriminates between models. The discriminatory power of a scoring rule is large at $\theta$ in the direction $\delta$ when, for some small $t>0$,
	\begin{equation}\label{eq:discrimination}
		S(\mathbb{Q}_\theta,\mathbb{Q}_\theta) - S(\mathbb{Q}_{\theta+ t\delta},\mathbb{Q}_\theta) \approx t^2\delta^T s(\mathbb{Q}_{\theta}) \delta
	\end{equation}
	is large. In the case of varying scale, where we have a set of probability measures $\{\Q_{\theta_i}\}$, the notion of a ``small'' perturbation should be viewed relative to the scale of the distribution for each $i$. If we would not do this,  the location and scale of $\Q_{\theta_i+\delta t}$ and of $\Q_{\theta_i}$ can be orders of magnitude apart for some $i$ even if $t>0$ is small. Thus, to study the discriminatory power of a composite scoring rule, it is natural to consider 
	\begin{equation}\label{eq:discComb}
		\frac{1}{n}\sum_{i=1}^n [S(\mathbb{Q}_{\theta_i},\mathbb{Q}_{\theta_i}) - S(\mathbb{Q}_{\theta_i+\sigma_i \delta t},\mathbb{Q}_{\theta_i})] \approx t^2 \delta^T \left( \frac{1}{n} \sum_{i=1}^n  \sigma_i^2 s  \left(\mathbb{Q}_{\theta_i} \right) \right) \delta.
	\end{equation}
	
	Using a scale invariant scoring rule means that it will have an equal discriminatory power for each observation in \eqref{eq:discComb}, since  $\sigma_i^2 s(\mathbb{Q}_{\theta_i})=  s(\mathbb{Q})$. For the CRPS, one instead has $\sigma_i^2 s(\mathbb{Q}_{\theta_i})= \sigma_i s(\mathbb{Q})$ (see Proposition~\ref{prop:kernel_scale}), hence the score puts more discriminatory power on observations with larger variability (the same holds for the MSE and the MAE). For the Hyvärinen score (if $\Q$ is a Normal distribution) then $\sigma_i^2 s(\mathbb{Q}_{\theta_i})= \frac{1}{\sigma^2_i} s(\mathbb{Q})$, and the score hence puts more discriminatory power on observations where the variability is small. One can thus see that the scale function captures the estimation properties of the scores shown in Figure~\ref{fig:bivar}, where the CRPS gives a large weight to the observation with a large variance and the Hyvärinen score gives a large weight to observation with a small variance.
	
	This observation about discriminatory power goes beyond the location-scale parameter setting. Suppose, for example, that we have an $AR(p)$ processes, $X_t$, and consider one-step-ahead predictions so that the $t$:th prediction has location parameter
	$
	\mu_t = \sum_{i=1}^p \alpha_{i} X_{t-i}.
	$
	If the innovation process has a non-constant variance, or if we do not start the process in the stationary distribution, then the scale parameter of the predictions will vary. Thus, if we use an average score over the different predictions we expect the score ruling to give discriminatory power to observations as described above.
	
	These observations explain the results in the example in Subsection~\ref{sec:motivation}, which shows a situation where one does not want to use a scale dependent scoring rule.
	However, it is also important to note that there are situations where scale dependence can be a good property, if we indeed care more about the discriminatory power for observations with high variability.
	In any case, one should be aware of the choice one is making when choosing a scoring rule for ranking forecasts. For example, the CRPS punishes errors linearly in scale which implies that large errors (in absolute terms) get punished more, while the SCRPS punishes  error in relative terms. This is seen  clearly in Section~\ref{sec:NBR} and Figure~\ref{fig:pedestrian}.

\subsection{The relation to homogeneity}\label{sec:homogeneity}

A property related to local scale invariance is homogeneity, \changed{which should not be confused with the concept of homogeneity in the density function as introduced by \citet{Parry2012}}. 
A scoring rule $S$ 
is said to be homogeneous  of order $b$ if 
$S(\mathbb{P}_{0,\sigma}, \mathbb{Q}_{0,\sigma}) = \sigma^b S(\mathbb{P}, \mathbb{Q})$
for all $\sigma>0$ such that the scale transformations $\mathbb{P}_{0,\sigma}, \mathbb{Q}_{0,\sigma}$ of the probability measures $\mathbb{P}$ and $\mathbb{Q}$ are well-defined. Homogeneous scoring rules have the desirable property that their rankings of predictions are invariant under scaling (or change of unit of the input).  \citet{patton2011volatility} and \citet{efron1991regression} noted the importance of this property for the loss function in regression. It is easy to see that the CRPS is homogeneous of order one. 

When using scoring rules for \changed{ranking of forecasts} only the difference between the scores is relevant. 
\cite{nolde2017elicitability} noted that this allows one to weaken the  homogeneity assumption. To make this precise we introduce  the concept of difference homogeneity. A scoring rule is difference homogeneous of order $b$ \changed{on $\mathcal{Q}$ (see Definition~\ref{def:measuresF}) if there exists a function $h :  \mathbb{R} \times \mathbb{R} \rightarrow \mathbb{R}$, which may depend on $\mathcal{Q}$, such that}
\[
\changed{S(\mathbb{P}_{0,\sigma}, \mathbb{Q}_{0,\sigma}) = \sigma^b S(\mathbb{P}, \mathbb{Q}) + h(\sigma,b) \quad \forall\, \mathbb{P},\mathbb{Q}\in \mathcal{Q}.}
\]

One can relate this property to local scale invariance when keeping the location parameter fixed. Specifically, if the scoring rule has a scale function $s$, then for each $t>0$
\begin{align*}
	S(\mathbb{Q}_{(0,\sigma)},\mathbb{Q}_{(0,\sigma)})-	S(\mathbb{Q}_{(0,\sigma( 1 + t))}, \mathbb{Q}_{0,\sigma})   &= 
	\sigma^b S(\mathbb{Q},\mathbb{Q}) -	\sigma^b S(\mathbb{Q}_{(0, ( 1 + t))}, \mathbb{Q}) \\
	&=	\sigma^b s(\mathbb{Q}) t^2 + o(t^2),  
\end{align*}
\changed{Hence at least for the scale parameter, a difference homogeneous scoring rule of order $b$  has the scale function $s(\mathbb{Q}_{(0,\sigma)}) = \sigma^b s(\mathbb{Q})$.
	From this it is also easy to see that with respect to scaling (not location), a scoring rule that is difference homogeneous of order $b$ is locally scale invariant if and only if $b=0$. }


\section{Robustness of scoring rules}\label{sec:robust}
Besides scale dependence, another problematic scenario is if the scoring rule is sensitive to outliers in the data. In this case, the average score \eqref{eq:average_score} may be heavily affected by only a few predictions. 
Also the sensitivity to outliers can give unintuitive results when using average scores for \changed{rankings of forecasts}, which is illustrated in the following example.

\begin{example}\label{ex2}
	Suppose that we have two competing models which are used for prediction of the two real-valued variables $Y_1$ and $Y_2$. The first model (shown in black in Figure \ref{fig:bivar2}) has $\P_{1} = \pN(0,0.01^2)$ and $\P_{2} = \pN(5,0.8^2)$. The second model (shown in red) has $\P_{1} = \pN(0,0.1^2)$ and $\P_{2} = \pN(4.9,0.85^2)$. Assume that we observe $y_1 = 0$ and $y_2 = 0.5$ and compute the average CRPS, log-score, and SCRPS for each model. The results are shown in Table~\ref{tab:ex2}.
	The second model is chosen by the CRPS and by the log-score, even though the first model is clearly more accurate for the first variable and both models are similarly inaccurate for the second. 
	The SCRPS on the other hand chooses the first model. 
\end{example}

\begin{figure}[t]
	\begin{center}
		\begin{minipage}[t]{0.4\linewidth}
			\begin{center}
				\includegraphics[width=\textwidth]{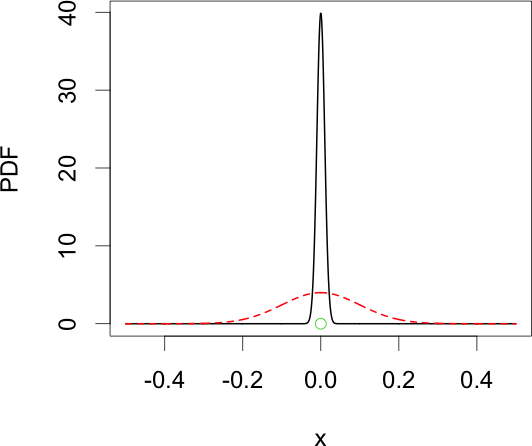}
			\end{center}
		\end{minipage}
		\begin{minipage}[t]{0.4\linewidth}
			\begin{center}
				\includegraphics[width=\textwidth]{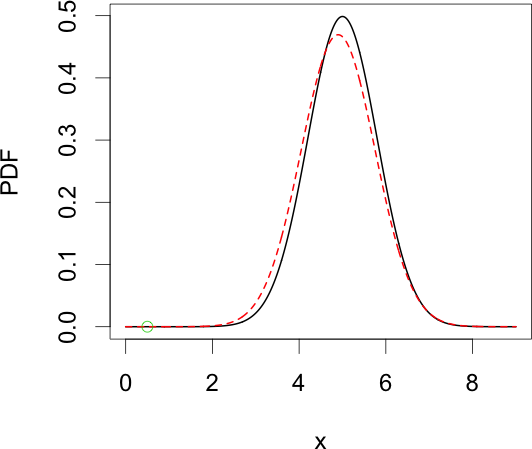}
			\end{center}
		\end{minipage}
		\vspace{-0.4cm}
	\end{center}
	\caption{Predictive distributions of two models, one shown in black and one in red, for two variables $Y_1$ (left) and $Y_2$ (right). Observations of each variable are shown in green.}
	\label{fig:bivar2}
\end{figure}

\begin{table}
	\caption{Results for Example~\ref{ex2}.}
	\centering
	\begin{tabular}{lcccccc}
		\toprule
		& \multicolumn{3}{c}{Model 1} & \multicolumn{3}{c}{Model 2}\\
		& CRPS & log-score & SCRPS & CRPS & log-score & SCRPS\\
		\cmidrule(r){2-4} \cmidrule(r){5-7}
		$Y_1$  	& -0.002	&  3.69 	&   1.53	& -0.02 & 1.38 & 0.38\\
		$Y_2$ 	& -4.05 	& -16.5	& -4.93	& -3.92 & -14.15 & -4.57\\
		mean 	& -2.02 	& -6.42 	& -1.70	& -1.97 & -6.38 & -2.09\\
		\bottomrule
	\end{tabular}
	\label{tab:ex2}
\end{table}

The apparent higher sensitivity to outliers of the log-score compared to CRPS can result in \changed{rankings} where models with larger variances are favored, even though this might not be optimal for most  locations. This higher robustness of the CRPS has previously been noted \citep{gneiting2007strictly}, but a natural question is if this is true also for other distributions than the normal. To assess this, we will study the behavior of $S(\P,y)$ as $|y|\rightarrow\infty$. The asymptotic rate of $S$ will in this case be a measure of the robustness of the scoring rule, and to get a scoring rule which is actually robust, we require that $S(\P,y)$ remains bounded as $y$ increases. Formally, we use the following definition of robustness of scoring rules. Here, we let $f(x) \asymp g(x)$ denote that there exist constants $0<c<C<\infty$ and $x_0>0$ such that $cg(x)\leq f(x) \leq Cg(x)$ for all $\|x\|>x_0$. 
\begin{Definition}\label{def:robust}
	A scoring rule $S$ on a normed space $(\Omega, \|\cdot\|_\Omega)$ is robust if $S(\P,y)$ is bounded as a function of $y$ for each $\P \in \mathcal{P}_S$. \changed{If there exists a number $\alpha_\P\geq 0$ such that $|S(\P,y)| \asymp  \|y\|_\Omega^{\alpha_\P}$ for a given $\P \in \mathcal{P}_S$, then we say that this scoring rule has model-sensitivity $\alpha_\P$. It if exists, the sensitivity index of $S$ is defined as $\alpha = \sup_{\P\in\mathcal{P}} \alpha_\P$. Thus, if $S$ has a sensitivity index $\alpha$, it is robust if $\alpha=0$.}
\end{Definition}

For $\Omega = \mathbb{R}$ and a normal distribution $\P = \pN(\mu,\sigma^2)$, the log-score has model-sensitivity $\alpha_\P=2$. Hence, the log-score is not robust. Using the expression of the CRPS for the normal distribution from Appendix \ref{app:gauss}, we get that the model-sensitivity in this case is $\alpha_\P=1$. Thus, the CRPS has a lower sensitivity than the log-score in the Gaussian case, but it is not robust. In fact, it is not always the case that the model-sensitivity is lower for the CRPS compared to the log-score. A simple counter example is to take $\P$ as the Laplace distribution, with $|\logs(\P,y)| \asymp y$ and $|\crps(\P,y)|\asymp y$.

\changed{
	\begin{Remark}
		The boundedness of $S(\P,y)$ might not be strictly necessary for applications, and one could alternatively consider a relaxed notion of robustness by assuming that the scoring rule is robust if  $|S(\P,y)| \asymp  \log( \|y\|_\Omega)$. Also note that $S$ might be bounded for all $\P\in\mathcal{P}_S$ even if there does not exists an $\alpha_\P$ for some $\P$. Thus, the existence of $\alpha$ in Definition~\ref{def:robust} is not necessary for robustness of $S$.
	\end{Remark}
	It is not difficult to see that our notion of robustness is equivalent to requiring that the scoring rule has a bounded influence function, which is a classical general definition of robustness \citep{hampel1974influence}. There are naturally other possible definitions of robustness that could be considered. In particular, proper scoring rules are often used for parameter estimation \citep[see, e.g.,][]{gneiting2007strictly,Dawid2016}, and for this scenario one could alternatively use the classical notion of B-robustness for M-estimators, which was discussed in the context of parameter estimation via scoring rules by \citet{Dawid2016}. However, since our focus is not on the use of scoring rules for parameter estimation, we will leave comparisons of our robustness notion with B-robustness for future work.}

\section{Kernel scores and robustness}\label{sec:kernel_scores}

The CRPS is a special case of the larger class of kernel scores, coined by \cite{dawid2007geometry}, which are created using a negative definite kernel. A real-valued function $g$ on $\Omega \times \Omega$, where $\Omega$ is a non-empty set, is said to be a negative definite kernel if it is symmetric in its arguments and if $\sum_{i=1}^n\sum_{j=1}^n a_i a_j g(x_i,x_j) \leq 0$ for all positive integers $n$, all $a_1,\ldots,a_n \in \mathbb{R}$ such that $\sum_{i=1}^n a_i=0$, and all $x_1,\ldots,x_n \in \Omega$. 
Given a negative definite kernel, the kernel score is created as in the following theorem by \cite{gneiting2007strictly}. 

\begin{Theorem}\label{thm:kernel}
	Let $\P$ be a Borel probability measure on a Hausdorff space $\Omega$. Assume that $g$ is a non-negative, continuous negative definite kernel on $\Omega \times \Omega$ and let $\mathcal{P}$ denote the class of Borel probability measures on $\Omega$ such that  $\E_{\P,\P}\left[ g(X,Y) \right] < \infty$. Then 
	\begin{equation}\label{eq:kernel_score}
		\Sker_g(\P, y) := \frac{1}{2} \E_{\P,\P} \left[g(X,Y) \right] - \E_{\P} \left[g(X,y) \right]
	\end{equation}
	is a proper scoring rule on $\mathcal{P}$. 
\end{Theorem}

One example of a family of negative definite kernels that can be used for $\Omega = \mathbb{R}$ is $g_{\alpha}(x,y) = |x-y|^{\alpha}$ for $\alpha\in (0,2]$, and we introduce the shorthand notation $\Sker_{\alpha}(\P,y)$ for this choice. The CRPS is the special case $\Sker_1(\P,y)$.

We are now interested in whether the kernel scores are locally scale invariant or robust. The robustness naturally depends on the properties of the kernel $g$.  Specifically, we have the following theorem.

\begin{Theorem}
	\label{the:kernelrob}
	Let $\P$ be a Borel probability measure on a normed vector space $(\Omega, \|\cdot\|_\Omega)$. Assume that $g$ is a non-negative, continuous negative definite kernel on $\Omega \times \Omega$, such that  $g(x,y)=g_0(\|x-y\|_\Omega)$, with $g_0(x) \asymp |x|^\alpha$ for some $\alpha>0$, and $|\E_{\P}[g_0(\|X\|_\Omega)]| < \infty$. Then $|S^{ker}_g(\P,y)| \asymp \|y\|_\Omega^\alpha$.
\end{Theorem}

\begin{proof}
	
	The result follows if we can show that the first term in \eqref{eq:kernel_score} is finite and that the second term decays as $-\|y\|_\Omega^\alpha$ when $y \rightarrow \infty$.
	For $x,y \in \Omega$ we have, by the definition of a negative definite kernel, that
	\begin{equation}\label{eq:kernel_rule}
		\begin{split}
			2 c_1 c_2 g(x,y) \leq& - c_1^2 g(y,y) - c_2^2 g(x,x) - c_0^2 g(0,0) - 2c_2 c_0 g(y,0) - 2c_1c_0g(x,0)
		\end{split}
	\end{equation}
	for any constants $c_0, c_1, c_2$ with $c_1+c_2+c_0=0$. Choose $c_1=c_2=1$ and $c_0=-2$, then 
	\eqref{eq:kernel_rule} can be reformulated as
	$g(x,y) \leq -g_0(0) +  2g_0(\|x\|_\Omega) + 2g_0(\|y\|_\Omega)$. Thus, we have $ |\E_{\P,\P} \left[g(X,Y) \right]| \leq |g_0(0)| +  4|\E_{\P}(g_0(\|X\|_\Omega))| < \infty$ since $\E_{\P}[g_0(\|X\|_\Omega)] < \infty$ and 
	\begin{equation}\label{eq:lowerbound}
		\Sker_g(\P, y) \geq \E_{\P,\P} \left[g(X,Y) \right]  + g_0(0) -  2\E_\P(g_0(\|X\|_\Omega)) - 2g_0(\|y\|_\Omega).
	\end{equation}
	
	If we now instead choose  $c_0=c_2=1$ and $c_1=-2$,   then
	\eqref{eq:kernel_rule} gives the lower bound
	$g(x,y) \geq 1.5g_0(0)  + g_0(\|x\|_\Omega) + 0.5g_0(\|y\|_\Omega)$. Thus
	\begin{equation}\label{eq:upperbound}
		\Sker_g(\P, y) \leq \E_{\P,\P} \left[g(X,Y) \right]  + 1.5g_0(0) -  \E_\P(g_0(\|X\|_\Omega)) - 0.5g_0(\|y\|_\Omega).
	\end{equation}
	Since  $g_0(\|x\|_\Omega) \asymp \|x\|_\Omega^\alpha$ and $\E_{\P}[g_0(\|X\|_\Omega)] < \infty$ the claim follows from \eqref{eq:lowerbound} and \eqref{eq:upperbound}. 
\end{proof}

The theorem shows that $\Sker_{\alpha}$ is not a robust scoring rule. However, we may modify the kernel in order to make it robust. For example:

\begin{Corollary}
	Let $\Omega = \mathbb{R}$ and for $c>0$ define
	\begin{equation}\label{eq:truncated_kernel}
		g_c(x,y) = \begin{cases}
			|x-y| & |x-y|<c,\\
			c & \mbox{otherwise.}
		\end{cases}
	\end{equation}
	Then the \textbf{robust CRPS} (rCRPS) defined as $\Sker_{1,c}(\P, y) := \Sker_{g_c}(\P, y)$ is a proper scoring rule on the class of Borel probability measures on $ \mathbb{R}$.
\end{Corollary}
\begin{proof}
	By the definition of a negative definite kernel, one has that $g(x,y)$ is negative definite if it can be written as $r(x,x) + r(y,y) - 2r(x,y)$ where $r(x,y)$ is positive definite. Here $g_c(x,y) = r_c(x,x) + r_c(y,y) - 2r_c(x,y)$ where $r_c(x,y) = 0.5(c - |x-y|)^+$ is the  positive definite triangular covariance function. 
	The result therefore follows from Theorem~\ref{thm:kernel} since $\pE_{\P,\P}[g_c(X,X)]\leq c$. 
\end{proof}

The constant $c$ here defines a limit where deviations are not further punished. Analytic expressions for this score in the case of the Gaussian distribution are given in Appendix~\ref{app:gauss}. Naturally, many other robust scoring rules can be constructed by replacing the triangular correlation function with some other compactly supported correlation function. 

The next question is whether the kernel scores, robust or not, are scale invariant. The following proposition shows that  $S^{ker}_{\alpha}$ is scale dependent.

\begin{Proposition}\label{prop:kernel_scale}
	\changed{Let $\Omega = \mathbb{R}$ and assume that $\mathcal{Q}$ is a class of  probability measures $\Q_{\theta}$ which are location-scale transformations satisfying Assumption~\ref{ass1} for $\alpha \in (0,2]$. 
		Then the scale function of $S^{ker}_{\alpha}$ on $\mathcal{Q}$ is
		$$
		s(\Q_{\theta})=\sigma^{\alpha-2}s(\Q,r) = \sigma^{\alpha-2}\E_{\Q,\Q}[H^{\alpha}_{\Q}(X,Y)],
		$$}
	where $H^{\alpha}_{\Q}(X,Y)$ is a $2\times 2$ matrix independent of $\theta$.
\end{Proposition}
The proof is given in Appendix~\ref{sec:proofs}.
For the robust CRPS we have the following result, which also is proved in  Appendix~\ref{sec:proofs}.
\begin{Proposition}\label{prop:scale_rob}
	\changed{Let $\Omega = \mathbb{R}$ and 
		assume that $\mathcal{Q}$ is a class of  probability measures $\Q_{\theta}$ which are location-scale transformations satisfying Assumption~\ref{ass1} for $\alpha=1$. Then the scale function of $S^{ker}_{1,c}$ on $\mathcal{Q}$ is
		$$
		s(\Q_{\theta})=\sigma^{-1} \E_{\Q,\Q}\left[ H_{\Q}(X,Y)\mathbb{I}\left(\left|X-Y\right|<\frac{c}{\sigma}\right)\right],
		$$}
	where $H_{\Q}(X,Y)$ is a $2\times 2$ matrix independent of $\theta$.
\end{Proposition}
This result implies that when applying $S^{ker}_{1,c}$ in a situation with varying $\sigma$, i.e., the true predictive distribution has varying scale, observations with large values of $\sigma$ will be less important since $c$ is fixed. Note that it is not possible to set $c$ to depend on $\theta$ since it is unknown. This implies that the robustness will protect against outliers that are large on an absolute scale of the predictive measure, but the robustness cannot protect against outliers for predictions where $\sigma$ is small, i.e., outliers on a relative scale.

In order to make the scoring rule robust against outliers of $\Q_{\theta}$ the bound would need to be scaled with $\sigma$, which as mentioned above is unknown a priori. It is to the authors' knowledge an open question how to create a proper scoring rule that protects against outliers in $\Q_{\theta}$. One option that could work in practice is to set $c$ dependent on some reference predictive distribution. Since $c$ should only protect against outliers it does not seems as invasive as scaling the actual scoring rule with a reference score, but is nevertheless still problematic.

\section{Generalized proper kernel scoring rules}\label{sec:standardized_scores}
In the previous section, we saw that one could make the kernel scores robust by adjusting the kernel, but that they in general are scale dependent. Because of this, we now want to construct a new family of scoring rules which can be made locally scale invariant. This is done in the following theorem, where the generalized proper kernel scores are introduced. 

\begin{Theorem}\label{thm:convexkernel}
	Let $\P$ be a Borel probability measure on a Hausdorff space $\Omega$. 
	Assume that $g$ is a non-negative, continuous negative definite kernel on $\Omega \times \Omega$ and that $h$ is a monotonically decreasing convex differentiable function on $\mathbb{R}^+$.  Then the scoring rule
	\begin{equation}\label{eq:scaled_score}
		\begin{split}
			S^h_g(\P, y) := &h\left(\E_{\P,\P}\left[g(X,Y)\right]\right) 
			+  2h'\left(\E_{\P,\P}\left[g(X,Y)\right]\right) \left(\E_{\P}\left[g(X,y)\right] - \E_{\P,\P}\left[g(X,Y)\right] \right)  
		\end{split}
	\end{equation}
	is proper on the class of Borel probability measures on $\Omega$ which satisfy $\E_{\P,\P}\left[ g(X,Y) \right] < \infty$.
\end{Theorem}

\begin{proof}
	For two measures $\P$ and $\Q$ in the class of Borel probability measures on $\Omega$ which satisfy $\E_{\P,\P}\left[ g(X,Y) \right] < \infty$, we need to establish that
	$S^h_g\left(\Q,\Q\right) \geq S^h_g\left(\P,\Q\right)$. By \citep[][Theorem 2.1, p.235]{van2012harmonic}, 
	\begin{align*}
		\E_{\P,\Q}\left[g(X,Y)\right] \geq \frac{1}{2}\left(\E_{\P,\P}\left[g(X,Y)\right] + \E_{\Q,\Q}\left[g(X,Y)\right]\right).
	\end{align*}
	Using this result and that
	$h'\left(\E_{\P,\P}\left[g(X,Y)\right]\right)<0$  it follows that
	\begin{align*}
		2h'\left(\E_{\P,\P}\left[g(X,Y)\right]\right)   & \left(\E_{\P,\Q}\left[g(X,Y)\right]- \E_{\P,\P}\left[g(X,Y)\right] \right) \leq \\
		&h'\left(\E_{\P,\P}\left[g(X,Y)\right]\right) \left(\E_{\Q,\Q}\left[g(X,Y)\right] - \E_{\P,\P}\left[g(X,Y)\right] \right).
	\end{align*}
	With this result and the fact that  $h$ is convex, we get
	\begin{align*}
		S^h_g(\P, \Q) \leq&  
		h\left(\E_{\P,\P}\left[g(X,Y)\right]\right)  + 
		h'\left(\E_{\P,\P}\left[g(X,Y)\right]\right)\left(\E_{\Q,\Q}\left[g(X,Y)\right] - \E_{\P,\P}\left[g(X,Y)\right] \right) \\
		\leq& h\left(\E_{\Q,\Q}\left[g(X,Y)\right]\right) = S^h_g(\Q,\Q).
	\end{align*}
\end{proof}

It should be noted that the idea of the theorem is similar to the construction of the supporting hyperplane in \cite[p.22]{Dawid1998}. Furthermore, the theorem could be generalized slightly by not requiring that $g$ is a continuous negative definite kernel but rather a function satisfying $g(x,y)\leq \frac12 \left(g(x,x) + g(y,y)\right)$. 

Clearly, the class of generalized proper kernel scores $S^h_g$ contains the regular proper kernel scores as a special case, obtained by choosing $h(x) = -x/2$. These scoring rules thus have the least convex (as it both convex and concave) and therefore fastest declining $h$ possible (up to a scaling factor). Further, choosing $g(x,y)=|x-y|$ and $h(x)=-\frac{1}{2} x$ for $\Omega = \mathbb{R}$ one obtains the regular CRPS. 
With $h(x) = -\frac{1}{2} \log(x)$, the proper scoring rule is given by
$$
S^{-\frac{1}{2} \log(x)}_g(\P, y) = -\frac{1}{2}\log(\E_{\P,\P}[g(X,Y)]) - \frac{\E_{\P}[g(X,y)]}{\E_{\P,\P}[g(X,Y)]} +1.
$$
Due to the second term in this expression, we will refer to it as the standardized kernel scoring rule with kernel $g$, denoted $\Ssker_{g}(\P,y)$ for short.

There are of course a myriad of options for $h$.  An interesting option is $h(x)=-\sqrt{x}$ which should act similarly to the standardized kernel score.
This flexibility is important since it allows for a wide range of generalized entropy terms, corresponding to different penalties for a priori uncertainty, while remaining proper scoring rules, see Appendix~\ref{sec:entropy}.

As for the usual kernel scores, $g_{\alpha}(x,y) = |x-y|^{\alpha}$ for $\alpha\in (0,2]$ is a natural choice of kernel for $\Omega = \mathbb{R}$, and we introduce the shorthand notation $\Ssker_{\alpha}(\P,y) := \Ssker_{g_\alpha}(\P,y)$ for the corresponding standardized kernel scoring rule. The special case $\Ssker_{1}(\P,y)$ is interesting since it provides a standardized analogue to the CRPS, which is the motivation for the SCRPS in \eqref{eq:scrps}. Using  that $\scrps = 1 + \Ssker_{1}$, where $\Ssker_{1}(\P,y)$ is proper by Theorem~\ref{thm:convexkernel}, we get:
\begin{Corollary}\label{cor:scrps}
	The SCRPS defined in \eqref{eq:scrps}  is a proper scoring rule on the class $\mathcal{P}_1$ of Borel probability measures on $\Omega = \mathbb{R}$ with $\E_{\P}\left[ |X| \right] < \infty$.
\end{Corollary}

Note also that \begin{align*}
	\Ssker_2(\P,y) &= 
	-\frac{1}{2}\log(\E_{\P,\P}[(X-Y)^2]) - \frac{\E_{\P}[(X-y)^2]}{\E_{\P,\P}[(X-Y)^2]} +1 \\
	&=
	-\frac{\left(y-\E_{\P}\left[Y\right]\right)^2}{2\V_{\P}\left[Y\right]} - \frac{1}{2} \log\left(\V_{\P}\left[Y\right] \right) + \frac1{2}(1-\log(2)),
\end{align*}
which is the Dawid-Sebastiani score \citep{DS1999} up to an additive constant.
The following proposition shows the local scale invariance of $\Ssker_{\alpha}$ for any $\alpha \in (0,2]$.

\begin{Proposition}
	\label{prop:scale_stand}
	Let $\Omega=\mathbb{R}$ and assume that $\mathcal{Q}$ is a class of  probability measures $\Q_{\theta}$ which are location-scale transformations satisfying Assumption~\ref{ass1} for $\alpha \in (0,2]$. 
	\changed{Then 
		$\Ssker_{\alpha}$ 
		has scale function $s(\Q_{\theta}) = \frac{1}{\sigma^2}\E_{\Q,\Q}[H_{\Q}^{\alpha}(X,Y)]$ on $\mathcal{Q}$,} where $H^{\alpha}_{\Q}(X,Y)$ is a $2\times 2$ matrix independent of $\theta$.
\end{Proposition}
The proof is given in Appendix~\ref{sec:proofs}.
The robustness of the standardized kernel scores is clearly equal to that of the corresponding kernel score. Therefore, $\Ssker_{\alpha}$ has the same robustness properties as $\Sker_{\alpha}$. We formulate this as a theorem.
\begin{Theorem}\label{thm:sta_rob}
	Let $\P$ be a Borel probability measure on a normed vector space $(\Omega,\|\cdot\|_\Omega)$. Assume that	 $g$ is a non-negative, continuous negative definite kernel on $\Omega \times \Omega$, such that  $g(x,y)=g_0(\|x-y\|_\Omega)$, with $g_0(x) \asymp |x|^\alpha$ for some $\alpha>0$, and $\E_{\P}[g_0(\|X\|_\Omega)] < \infty$. 
	Then $|\Ssker_{g}(\P,y)| \asymp\|y\|_\Omega^\alpha$.
\end{Theorem}
The proof is omitted as it is almost identical to that of Theorem \ref{the:kernelrob}.

Another interesting scoring rule for $\Omega=\mathbb{R}$ is the standardized kernel score which uses the kernel \eqref{eq:truncated_kernel}, which we denote as rSCRPS or $\Ssker_{1,c}$, where $c$ is the constant in the function $g_c$. It could be thought of as a robust version of the SCRPS, but it should be noted that it cannot be locally scale invariant by the same reasons as for the rCRPS. We will however later use this in one of the applications as an option that protects against outliers but has better scaling properties than the CRPS.

So far, we have only considered how to formulate scale invariant versions of kernel scores. However, there exists several other popular scoring rules, such as the continuous ranked logarithmic score \citep{juutilainen2012exceedance,todter2012generalization},
which are not defined through kernels like the CRPS. 
It might not be clear whether they are scale invariant, and Theorem~\ref{thm:convexkernel} cannot be directly used to create standardized versions. However, in those cases, the following result can instead be used. The theorem defines a transformation of a negative proper scoring rule, that is still a proper scoring rule and which, at least intuitively, should be less scale dependent.
\begin{Theorem}\label{thm:general_scale}
	Let $\Omega$ be a Hausdorff space and let $S$ denote a scoring rule that is  proper on a class of Borel probability measures $\mathcal{P}\subset\mathcal{P}_S$ on $\Omega$ such that $S(\P,y) <0$ for all $\P\in\mathcal{P}$ and $y\in\Omega$. Then
	\begin{align*}
		\Strans_S(\P,y) := \frac{S(\P,y)}{|S(\P,\P)|} -  \log(|S(\P,\P)|)
	\end{align*}
	is also a proper scoring rule on $\mathcal{P}$. Further, if $S$ is strictly proper on $\mathcal{P}$, then so is $	\Strans_S$.
\end{Theorem}
\begin{proof}
	Let $\Q, \P \in \mathcal{P}$.
	Since $S$ is a proper scoring rule, we have
	\begin{align*}
		\Strans_S(\P,\Q)=	\frac{S(\P,\Q)}{|S(\P,\P)|}  -  \log(|S(\P,\P)|) \leq  \frac{S(\Q,\Q)}{|S(\P,\P)|}  -  \log(|S(\P,\P)|).
	\end{align*}
	If $S$ is strictly proper the inequality is strict unless $\P=\Q$.
	Now since for $c<0$  the function $\frac{c}{x}-\log(x)$ attains its maximum value ($-1 - \log |c|$) at $x=-c$ 
	it follows that
	\begin{align*}
		\frac{S(\Q,\Q)}{|S(\P,\P)|} - \log(|S(\P,\P)|)\leq  \frac{S(\Q,\Q)}{|S(\Q,\Q)|} - \log(|S(\Q,\Q)|) = \Strans_S(\Q,\Q).
	\end{align*}
\end{proof}

\changed{One can note that the Dawid-Sebastiani score is obtained by applying Theorem~\ref{thm:general_scale} to the MSE $S(\mathbb{P},y) = (\pE_{\P}[X] - y)^2$. If one instead applies} the theorem to the CRPS, the result is a constant plus two times the SCRPS defined in \eqref{eq:scrps}. This serves as another motivation for why the SCRPS can be seen as a standardized version of the CRPS. Further, the theorem strengthens Corollary~\ref{cor:scrps} since it shows that the SCRPS in fact is strictly proper on $\mathcal{P}_1$.   

\begin{table}
	\caption{\label{tab:summary} Summary of local scale dependence and robustness properties for various scoring rules. }
	\centering
	\begin{tabular}{lcc}
		\toprule
		Scoring rule &  Locally scale invariant & Robust\\
		\midrule
		CRPS 					 & No & No\\
		SCRPS 					& Yes & No\\
		rCRPS 					& No & Yes\\
		rSCRPS					& No & Yes\\
		log-score				& Yes & No \\
		Dawid-Sebastiani	& Yes & No\\
		Hyvärinen				& No & No \\
		\bottomrule
	\end{tabular}
\end{table}
\changed{We end this section by summarizing the scale dependence and robustness properties of the various scores that we have discussed in Table~\ref{tab:summary}. In the table, the robustness of the log-score and the Hyvärinen score depends on the distribution.}

\section{Applications}\label{sec:application}

\subsection{A stochastic volatility model}
In this application, we want to highlight the difference between locally scale invariant and scale dependent scoring rules when there is variability in the scaling of the data which is caused by the model. Consider the following stochastic volatility model \citep{shephard1994partial},
\begin{align*}
	X_t &= a X_{t-1} + \epsilon^X_t, \quad t = 1, 2,\ldots, \\
	y_t &= \epsilon_t^Y \exp(X_t),
\end{align*}
where $\epsilon^X_t \sim \pN(0,\sigma^2_X)$ and $\epsilon^Y_t \sim \pN(0,\sigma^2_Y)$, with $a=0.95$, $\sigma_Y=1$, and $\sigma_X=0.5$. 

\begin{figure}[t]
	\centering
	\includegraphics[scale=0.3]{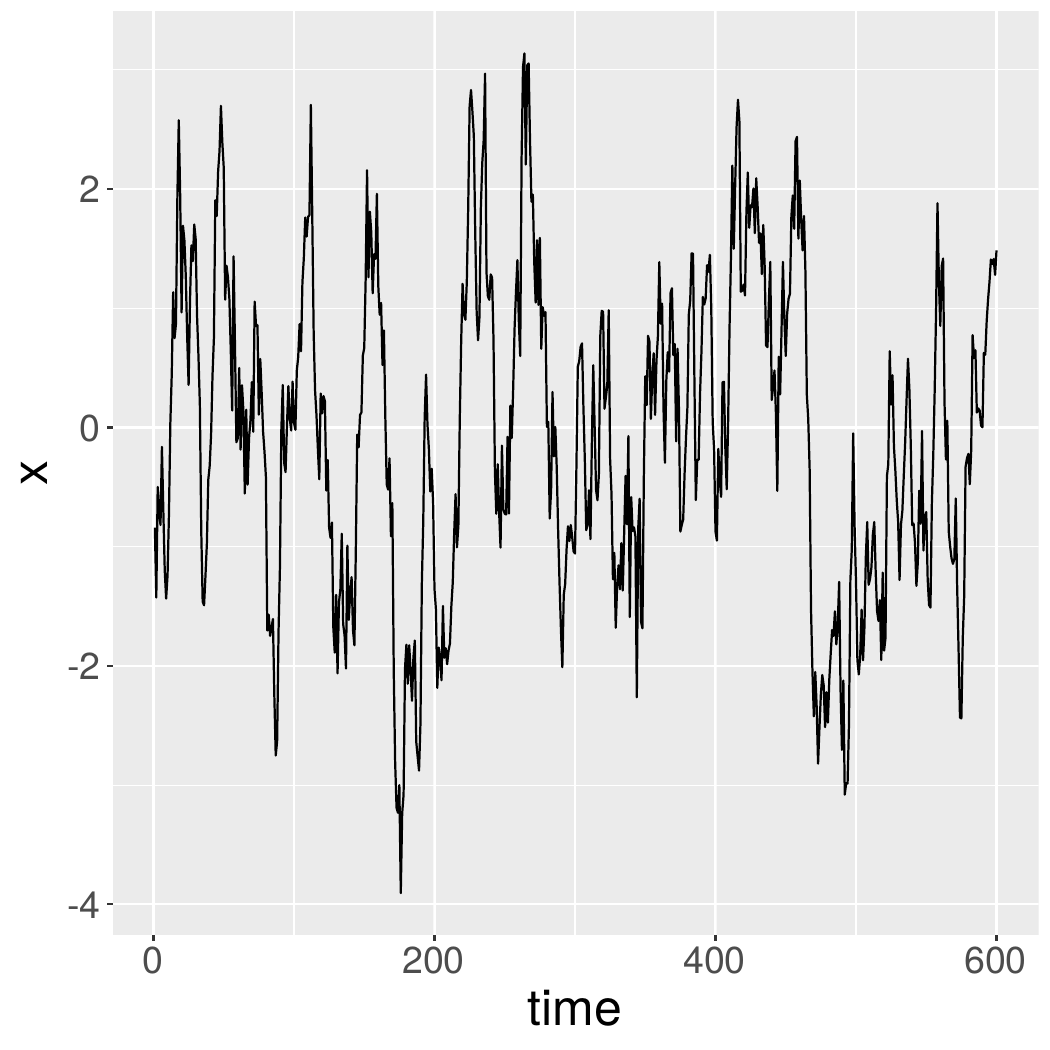}
	\includegraphics[scale=0.3]{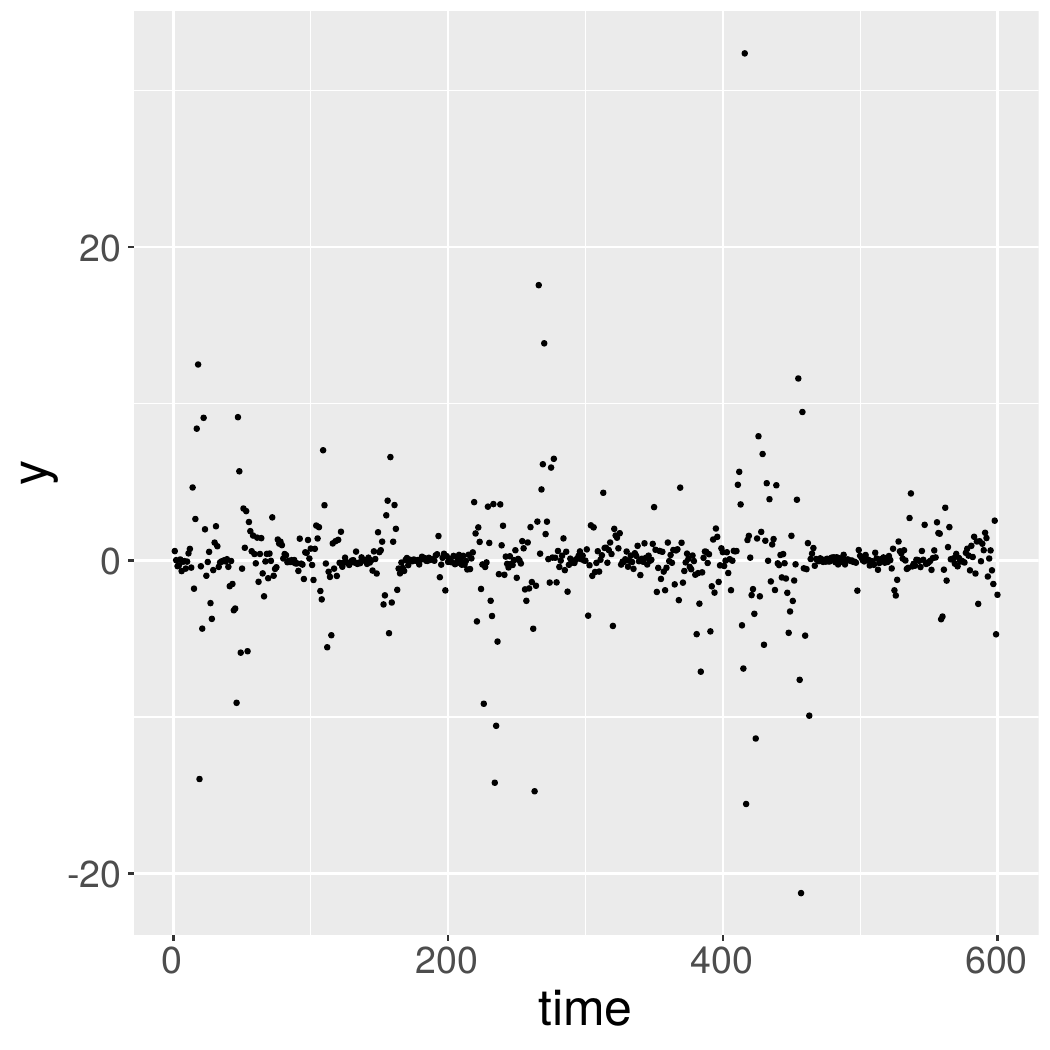}
	\caption{A realization of the volatility $x$ (left), and the observations of the model $y$ (right).}
	\label{fig:SimVol}
\end{figure}

Figure \ref{fig:SimVol} displays realizations of $Y_t$ and $X_t$. 
Although the observations are equally spaced, the varying volatility will result in that the proper scoring rules will put different importance on the different observations. We explore the model under a simplified assumption that we observe $X_t$ and want to predict $y_t$, which simplifies the analysis without altering the message that we want to convey.

To see how the stochastic volatility affects \changed{forecast rankings}, we compare how often the average score for each scoring rule is higher for the correct model compared to models with misspecified $\sigma_Y$. We simulate $500$ different realizations of the volatility model, where each simulation is a time series of length $600$. The left plot of Figure \ref{fig:SelectVol} shows the percentage of these realizations where the correct model, with $\hat{\sigma}_Y=\sigma_Y$, is chosen (has the largest \changed{average} score) for the  different scoring rules. The right plot of Figure \ref{fig:SelectVol}  shows  the percentage of these realizations where the correct model was significantly better then the alternative using a Diebold--Mariano test \citep{diebold1995paring}. As alternative models in the comparison, one with $\hat{\sigma}_Y = \sigma_Y+ \Delta$ and one with $\hat{\sigma}_Y = \sigma_Y - \Delta$ are used. The figure shows the results as  functions of $\Delta \in (0,0.5)$. \changed{Note that the SCRPS and the log score are virtually identical whereas both the CRPS and the Hyvärinen performs considerably worse in finding the true parameter. This pattern is even stronger when considering how often the model with true parameter is significantly better than the others.}
Further comparisons for this example can be seen in Appendix~\ref{sec:entropy}.
It should be noted here that, since all scores are proper, the probability of choosing the correct model will converge to one as the length of the series goes to infinity. However these pre-asymptotic differences can have a big effect on applications with limited amounts of data.


\begin{figure}
	\centering
	\includegraphics[scale=0.19]{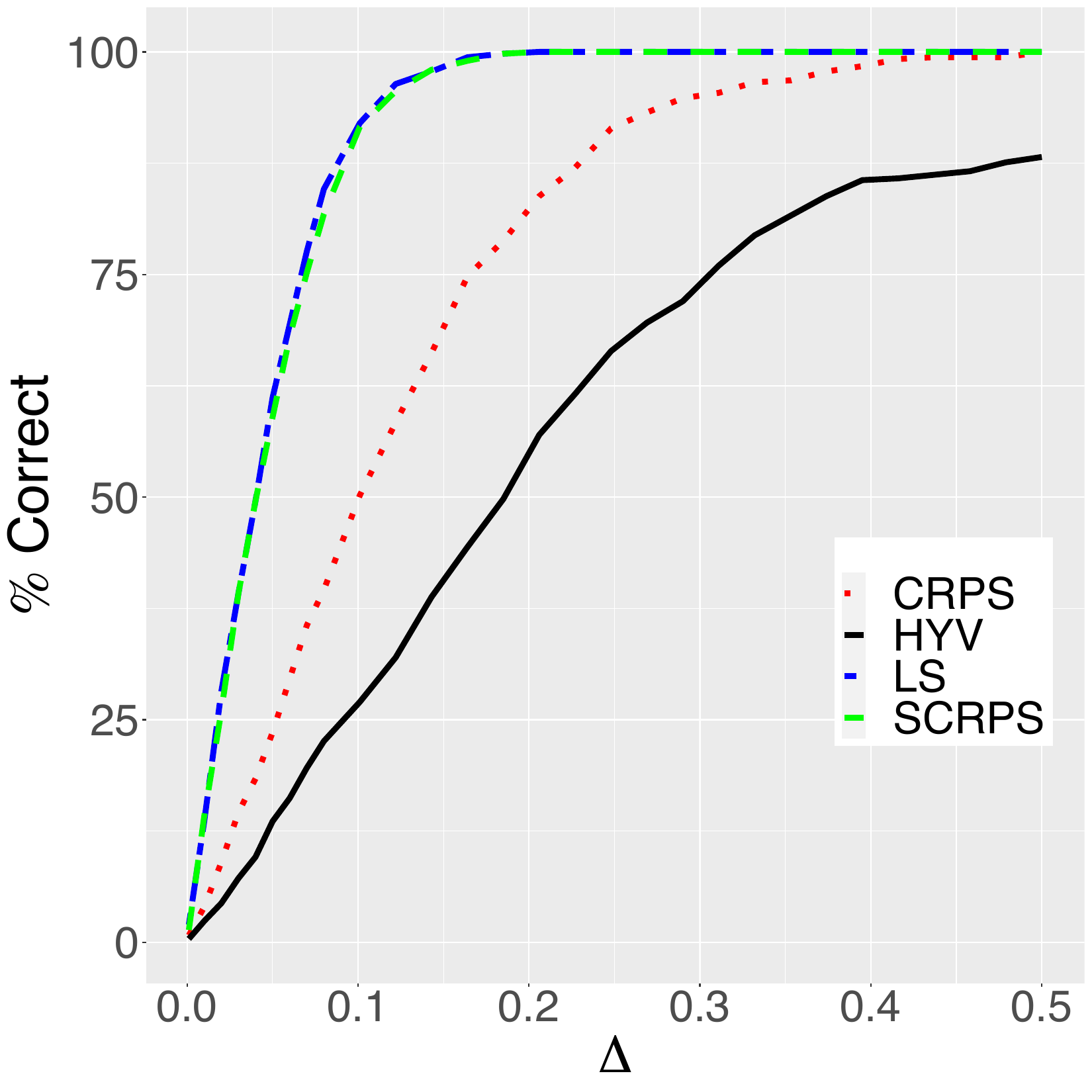}
	\includegraphics[scale=0.19]{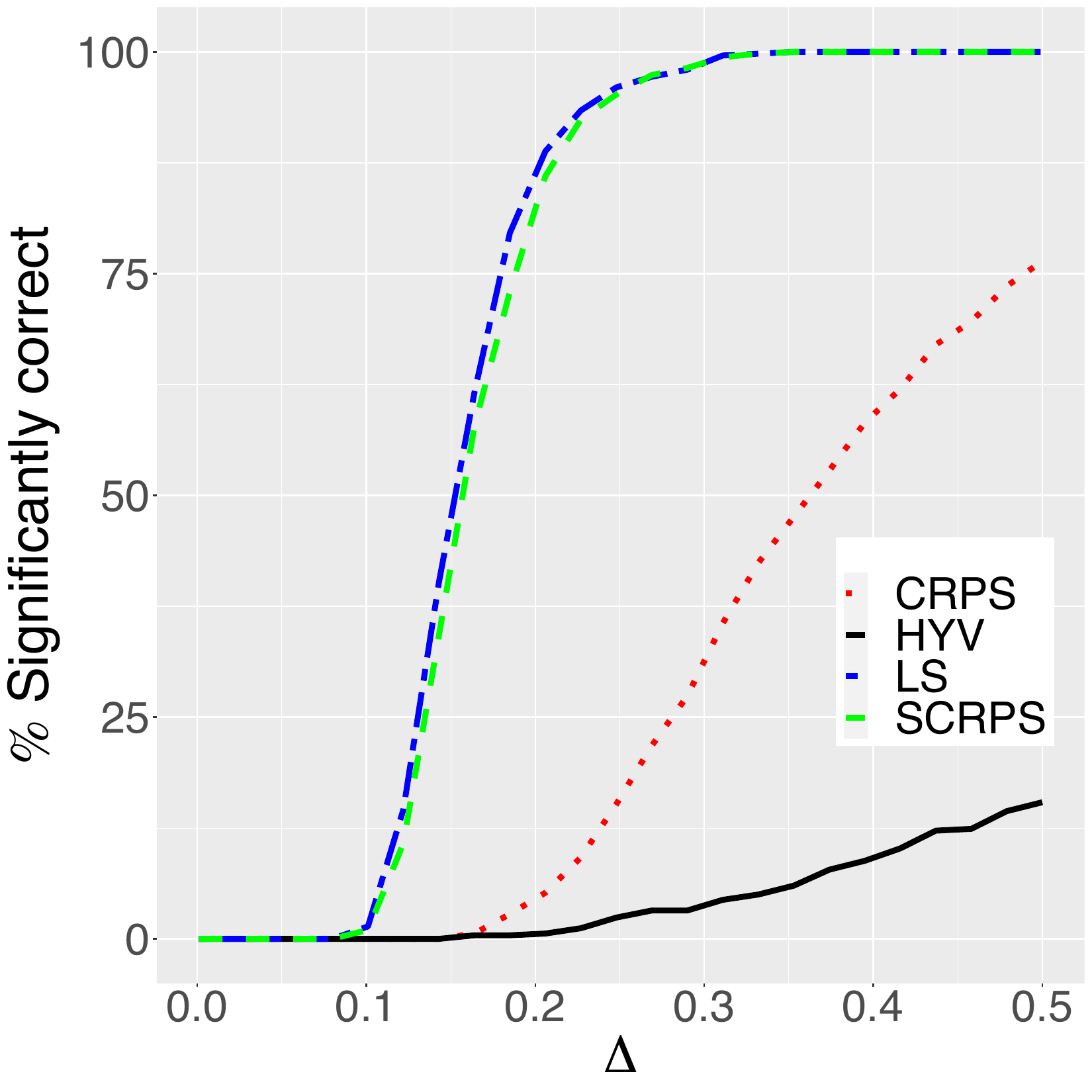}
	\caption{\changed{Percentages of times the true model, with $\hat{\sigma}_Y=\sigma_Y$ for $\sigma_Y=1$, is selected over two alternative models model with $\hat{\sigma}_Y = \sigma_Y \pm \Delta$, for four different scoring rules (left). The right panel shows the percentages of times the true model is found to be significantly better than the alternatives using a DM test with $\alpha=0.05$.
			Note that curves go to zero as $\Delta\rightarrow 0$ for each scoring rule since for each fixed set of observations, the value of $\hat{\sigma}_Y$ that minimizes the average score  is never exactly equal to $\sigma_Y$, and if it is outside the interval $[\sigma_Y - \Delta,  \sigma_Y + \Delta]$ then one of the alternative models will have the best average score. }  }
	\label{fig:SelectVol}
\end{figure}

\subsection{An application from spatial statistics}
\label{sec:spatstat}
A common use of proper scoring rules is to evaluate the predictive performance of random field models in spatial statistics. As an illustration of this, we generate $n=100$ observations $y_i = X(\mv{s}_i)$, $i=1,\ldots, n$ of a mean-zero Gaussian random field with a Mat\'ern covariance function 
$$
\Cov(X(\mv{s}),X(\mv{s}')) = C(\|\mv{s}-\mv{s}'\|;\mv{\Psi}) = \frac{\sigma^2}{2^{\nu-1}\Gamma(\nu)} (\kappa \|\mv{s}-\mv{s}'\|)^{\nu} K_{\nu}(\kappa \|\mv{s}-\mv{s}'\|).
$$
Here $\mv{\Psi} = (\kappa,\sigma,\nu)$ is a vector with the parameters of the model which we set to values that could occur in applications. Specifically, we use 
$(50, 1, 3)$ so that the field has variance $1$, is two times mean-square differentiable, and has a practical correlation range of approximately $0.1$. A simulation of the model is shown in Figure~\ref{fig:spatial_sim}, which also shows an example of the observation locations $\mv{s}_i$ drawn at random from a uniform distribution on $[0,1]\times [0, 1]$. \changed{For more details about spatial statistics and spatial prediction using Gaussian random fields, see, e.g., \citet[][Chapter 2]{Handbook2010}. }

\begin{figure}[t]
	\centering
	\includegraphics[width=0.45\linewidth]{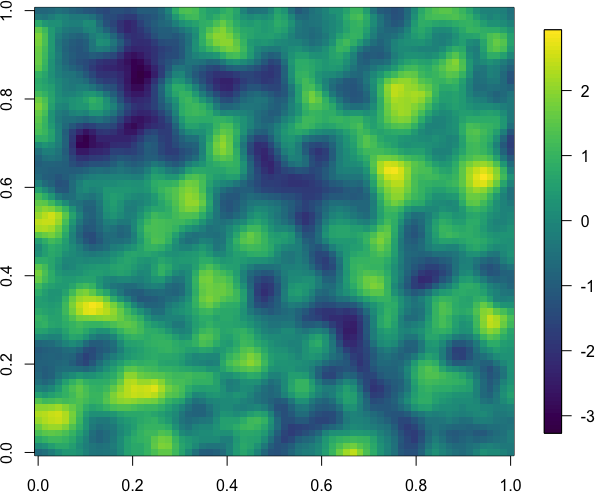}
	\includegraphics[width=0.39\linewidth]{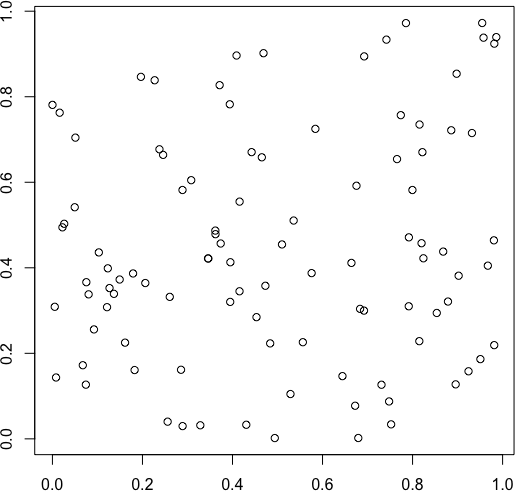}
	\caption{A Gaussian random field simulation (left) and observation locations (right).}
	\label{fig:spatial_sim}
\end{figure}

A measure of the predictive quality of a model with parameters $\mv{\Psi}^*$ is the average score \eqref{eq:average_score} in a leave-one-out cross-validation scenario. That is, $\P_{i} = \pN(\mu_i(\mv{\Psi}^*), \sigma_i^2(\mv{\Psi}^*))$ is the conditional distribution of $X_i$ given given all data except that at location $\mv{s}_i$, which we denote by $\mv{y}_{-i}$. If we let $\mv{\Sigma}$ denote the covariance matrix of $\mv{y}_{-i}$, and let $\mv{c}$ be a vector with $n-1$ elements $\{C(\|s_i - s_j\|;\mv{\Psi}^*), j\neq i\}$, the parameters of the predictive distribution are $\mu_i(\mv{\Psi}^*) = \mv{c}^T \mv{\Sigma}^{-1}\mv{y}_{-i}$ and $\sigma_i^2(\mv{\Psi}^*) = \sigma^2 - \mv{c}^T \mv{\Sigma}^{-1}\mv{c}$.
One can note that $\sigma_i^2(\mv{\Psi}^*)$ depends on the spatial configuration of the observation locations, where prediction locations close to observation locations will have lower variances.

Scale dependence of scoring rules increases the variance of the values for large distances, which may result in a larger variance of the average score. That the average CRPS has a larger variance than the average SCRPS means that it could be more likely that $S_n(\mv{\Psi}) < S_n(\mv{\Psi}^*)$ even if the data is generated using the parameters $\mv{\Psi}$. That is, it is more likely that the incorrect model choice is made if the average score is used for model selection. This is illustrated in the top left panel of Figure~\ref{fig:spatialKappa}, which shows the proportion of times that $S_n(\mv{\Psi}) > S_n(\mv{\Psi}^*)$ as a function of $\Delta$ when $\mv{\Psi}^* = (\kappa\pm \Delta, \sigma, \nu)$ and $n=100$ observations is generated using the parameters $\mv{\Psi} = (\kappa, \sigma, \nu) = (50, 1, 3)$. The top right panel shows the same result with $n=200$ observations. In both cases, the results are shown for the CRPS, the SCRPS, and the log-score, as well as the robust versions of the CRPS and the SCRPS with a limit value $c=2$. This limit value is equal to two times the variance of the field and is thus a quite high value given that the predictive variances often will be much lower than this. The limit should therefore not affect the predictions except at the locations close to potential outliers. 
One can note that, compared to the CRPS and the robust CRPS, the log-score, the SCRPS, and the robust SCRPS more often make the correct model choice for a given value of $\Delta$. One can also note that the robust scores in this case perform similarly to the regular scores, since the value of $c$ is rather high and since there are no outliers. The results are based on $2000$ different simulations of the field $X$ and the observation locations when $n=100$ and on $1000$ simulations when $n=200$. 

To illustrate why the robust scores may be useful, we redo the same simulation study with only one difference: For one of the observations $y_i$, chosen at random, we add a $\pN(0,5^2)$ variable, which thus makes this observation an outlier that does not follow the assumed model. The lower row of Figure~\ref{fig:spatialKappa} shows the results for this case. One can note that the outlier makes it more likely to choose incorrect model, but that this effect is reduced if the robust scores are used. 
In summary, if one were to choose one scoring rule to use for this type of data, where outlier may be present, the robust SCRPS is likely a good choice since it performs well both with and without outliers. 

\begin{figure}[!t]
	\centering
	\includegraphics[width=0.4\linewidth]{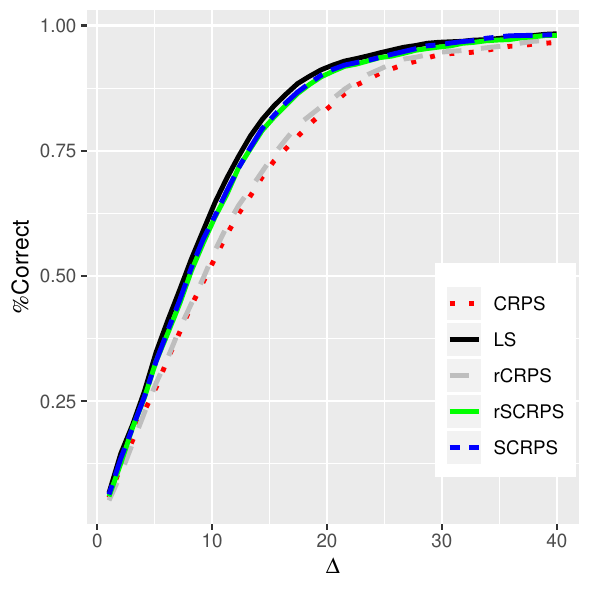}
	\includegraphics[width=0.4\linewidth]{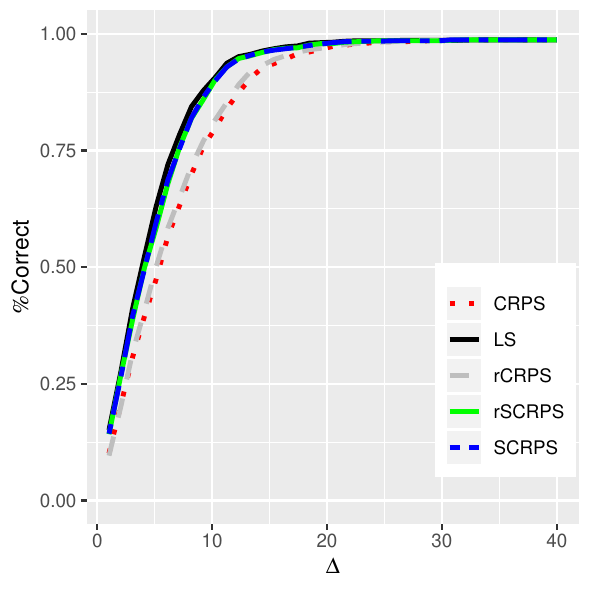}
	\includegraphics[width=0.4\linewidth]{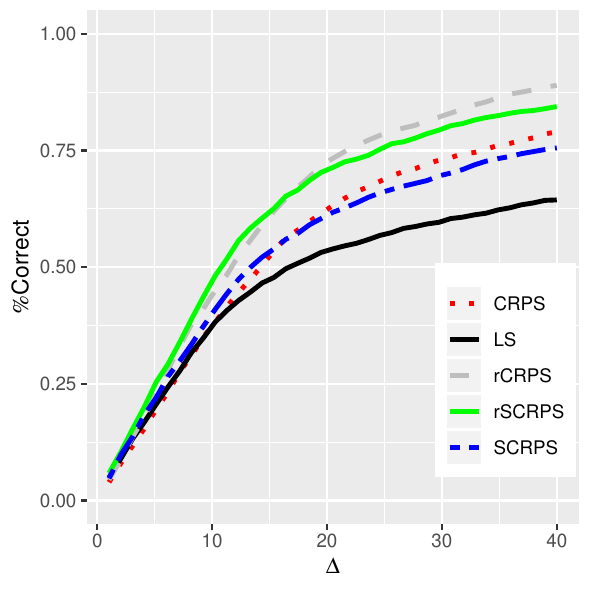}
	\includegraphics[width=0.4\linewidth]{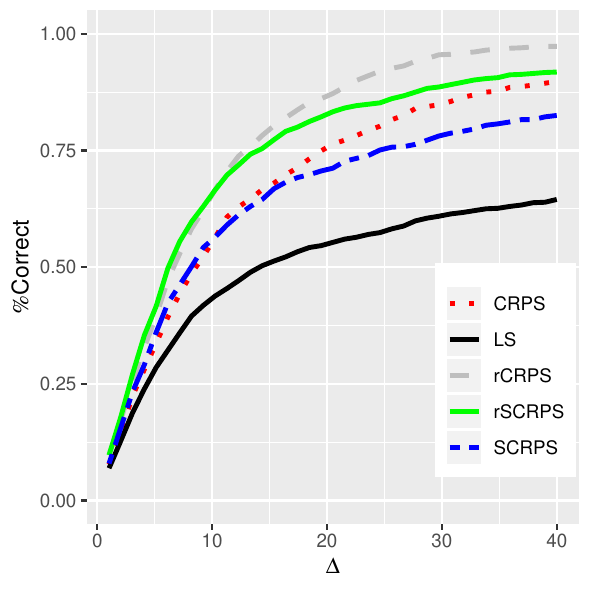}
	\vspace{-0.4cm}
	\caption{\changed{Percentages of times the correct model,with $\kappa = 50$, is selected} instead of models with $\kappa = 50 \pm \Delta$ for different scoring rules. The results are based on $n= 100$ observations and $2000$ simulations to the left, and $n = 200$ observations and $1000$ simulations to the right. In the top row, the data does not have any outliers, whereas it has one outlier in the bottom row. }
	\label{fig:spatialKappa}
\end{figure}

\subsection{Negative binomial regression}
\label{sec:NBR}
As a final application, we consider a negative-binomial regression model from an application in Space Syntax research \citep{hillier1993}. The application is described in detail in \citep{spacesyntax2019}. The data we consider consist of daily counts of the number of pedestrians that walked on 227 different street segments in Stockholm, Sweden. The data can be explored in the web application \citep{stepflow}, and we refer to   \citep{spacesyntax2019} and \citep{Pont2019a} for details about how the data were collected.
The goal is to explain the number of pedestrians walking on a given street through covariates in a regression model. If such a model fits well, it could for example be used to predict the number of pedestrians in new neighborhoods that are planned to be built in the city. Since the observations are counts, a reasonable model is a negative-binomial regression model. We assume that the observed count at street segment $i$ has a negative binomial distribution, $Y_i \sim \mbox{nBin}(\mu_i, s)$, where $\mu_i$ is the expected value of $Y_i$ and $s>0$ is a dispersion parameter that controls the variance of $Y_i$, which is $\pV(Y_i) = \mu_i + \mu_i^2/s$.  The mean is modeled as 
\begin{equation}\label{eq:regression}
	\log(\mu_i) = \sum_{k=1}^K \theta_k X_{k,i},
\end{equation}
where $X_{k,i}$ is the value of the $k$:th covariate at street segment $i$, and $\theta_k$ is the corresponding regression coefficient. We have nine covariates: 1) The weekday the measurement was taken; 2) The number of schools within 500m walking distance to the street segment;  The number of public transport nodes 3)  on the street segment, and 4) within 500m walking distance to the street segment; the number of local markets such as shops and caf\'es 5) on the street segment, and 6) within 500m walking distance to the street segment; as well as three covariates that related to the centrality of the street in the street network and the density of the buildings around the street \citep[see][for further details]{spacesyntax2019,Pont2019a}. 

\begin{figure}[t]
	\centering
	\includegraphics[width=0.8\linewidth]{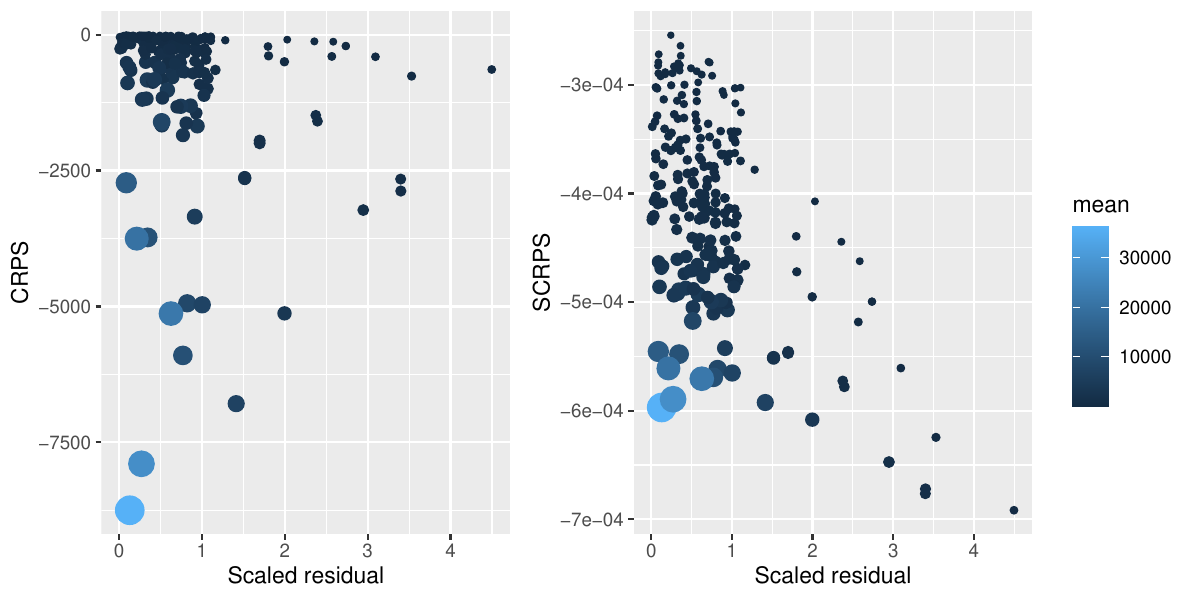}
	\includegraphics[width=0.8\linewidth]{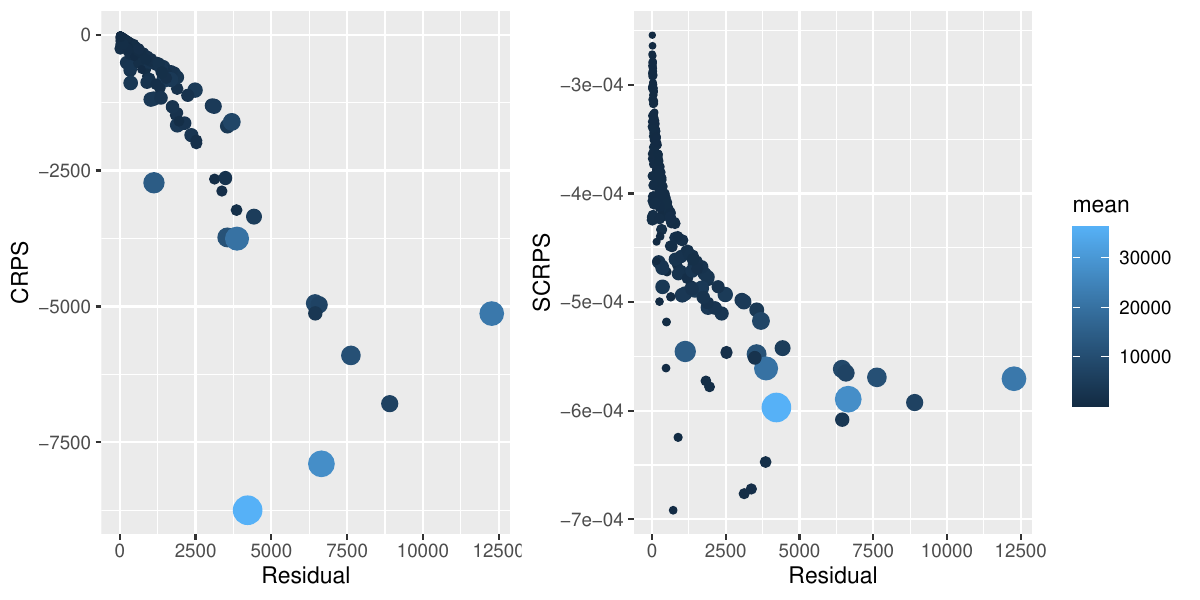}
	
	\caption{CRPS and SCRPS shown against scaled and unscaled residuals. The size and color of the points are determined by the expected value of the predictor.}
	\label{fig:pedestrian}
\end{figure}

We fit the model to the data using the R function \texttt{glm.nb} function from the MASS package \citep{MASSpackage}, and compute the CRPS and SCRPS value for each observation based on the model. 

In Figure \ref{fig:pedestrian}, the values of the scoring rules for each observation are plotted against the residuals, $e_i$ and the scaled residuals, $e^s_i$, defined as
\begin{align*}
	e_i =|y_i - \hat{\mu}_i|, \quad \mbox{and}\quad 
	e^s_i = \frac{e_i}{\hat{\sigma}_i} = \frac{|y_i - \hat{\mu}_i|}{\sqrt{\hat{\mu}_i+\hat{\mu}_i^2/\hat{s}}},
\end{align*}
where $\hat{s}$ is the estimated dispersion parameter and $\hat{\mu}_i$ is given by \eqref{eq:regression} with the estimated regression parameters. In the figure, the size and color of each observation is determined by $\hat{\mu}_i$, and one can note that the large (in absolute value) CRPS values mostly coincide with observations that have large residuals, while the large SCRPS values mainly occur for observations with large values of the standardized residuals.
For the residuals, one can see a linear tendency for the CRPS and a logarithmic tendency for the SCRPS. These are the, data free, entropy components of the scores -- $\log(E_{\P,\P}[|X-Y|])$ for SCRPS and $\E_{\P,\P}[|X-Y|]$ for CRPS-- which are given in detail in Appendix~\ref{sec:entropy}. This is a clear example of scale dependence, where the values for streets with high expected counts will be much more important than streets with low counts for the average CRPS whereas they will have a more even importance for SCRPS. 

Obviously one should be very careful with using average CRPS in this case, since a single observation might drastically change the average score. To illustrate this, we compute the average CRPS of the $k$ observations with smallest values of $\mu_i$ and divide this number for each $k$ by the average CRPS for all observations.  The result is shown as a function of $k$ in Figure \ref{fig:pedestrian2}, where we also show the same values for the SCRPS. One can note that removing around 20 of the observations with the largest values of $\mu_i$ reduces the average CRPS by half, whereas the SCRPS is much less sensitive to the removal of observations.

\begin{figure}[t]
	\centering
	\includegraphics[width=0.6\linewidth]{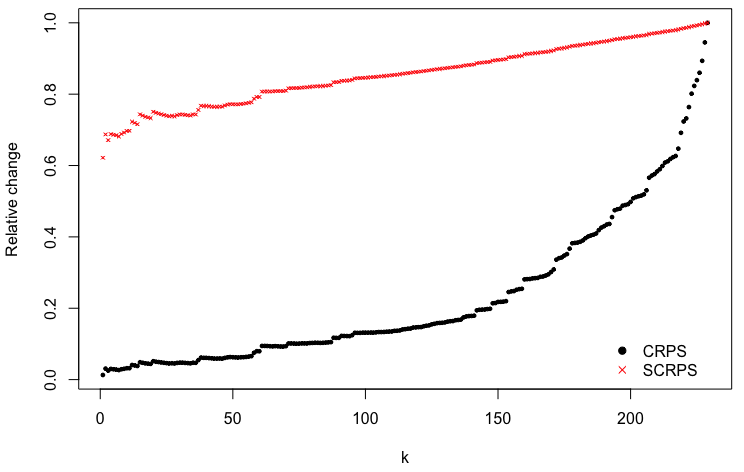}
	\vspace{-0.2cm}
	\caption{The average score of the $k$ smallest observations divided by the total average score.}
	\label{fig:pedestrian2}
\end{figure}

\section{Discussion}\label{sec:discussion}
We have illustrated how scoring rules such as the CRPS, \changed{the Hyv{\"a}rinen score,} the MSE, and the MAE are scale dependent and why this may lead to unintuitive model choices when used for forecast rankings and model selection. In such situations it may be better to use locally scale invariant scores such as the SCRPS from our proposed family of generalized proper  kernel scores, as we showed in three applications. 
An important property of the generalized proper kernel scores is that they are as easy to compute as the CRPS, and that they can be approximated through Monte Carlo integration in the same way as for the CRPS for models without analytic expressions for the scoring rules.  

An advantage with CRPS is that it allows for comparison with deterministic prediction models. This is not possible for SCRPS since it requires positive variance of the distributions to be finite. It further may cause problems for ensemble forecasts if the different forecasts happen to be equal (so that their empirical variance is zero). An alternative scoring rule that could be used in this case is a generalized kernel score with  $h(x)=-\frac{1}{2}\log(x+\gamma)$ for some $\gamma>0$ as $h$-function. The result of the scoring rule will then depended on the choice of the truncation parameter $\gamma$ which need to be set by the user, and we leave investigations of how to best do that for future research.

We defined the concept of local scale invariance in the case of location-scale transformations. This allowed for a relatively simple mathematical analysis of the problematic scale dependence that is encountered for the CRPS. This means that our definition of scale dependence was not strictly applicable in the final application with the negative-binomial distribution, since it does not have scale and location parameters. Nevertheless, we saw that the scaling issue remained the same. A natural extension is to extend the concept of local scale invariance to other classes of distributions, such as the negative binomial, that cannot be seen as location-scale transformations. 

Another issue that we have not addressed is that there is typically dependence between the predicted observations that are used in the average. How to take this dependence into account when comparing models is an interesting topic for future research, and we believe that scale dependence will be an important issue to consider also in that scenario. 

\changed{We introduced introduced the concept of robustness for scoring rules by requiring that $S(\P,y)$ remains bounded as $y\rightarrow\infty$ in order for the score to be robust. An interesting question for future research is whether it is possible to create a scoring rule that is both robust and locally scale invariant. Our conjecture is that this is not possible with our current notion of robustness. It would also be interesting to compare our notion of robustness with other classical definitions of robustness in the literature. As previously mentioned, proper scoring rules are often used for parameter estimation and for this scenario it will be interesting to compare our definition of robustness with the classical notion of B-robustness for M-estimators. It would also be interesting to  investigate the effect that local scale invariance has on parameter estimation in future work.}
%

%

%
\section*{Acknowledgements}
The authors would like to acknowledge \changed{the editors and the reviewers, as well as} Håvard Rue, Finn Lindgren, and Tilmann Gneiting for  helpful comments and suggestions that greatly improved the manuscript.

	\bibliographystyle{chicago}
	\bibliography{probPred}
	\begin{appendix}
		\section{Analytic expressions in the Gaussian case}\label{app:gauss}
For a Gaussian distribution we have 
$$
\Sker_{1}(\pN(\hat{\mu},\hat{\sigma}^2),y) = \frac{\hat{\sigma}}{\sqrt{\pi}} - 2\hat{\sigma}\varphi\left(\frac{y-\hat{\mu}}{\hat{\sigma}}\right) - (y-\hat{\mu})\left(2\Phi\left(\frac{y-\hat{\mu}}{\hat{\sigma}}\right) - 1\right),
$$
where $\varphi(x)$ and $\Phi(x)$ denotes the density function and cumulative distribution function of the standard normal distribution respectively \citep{gneiting2007strictly}. This follows directly from the definition of the score as a kernel score with kernel $g(x,y) = |x-y|$ in combination with the fact that 
\begin{equation}\label{eq:fnorm}
\pE_{\pN(\mu,\sigma^2)}[|X|] = 2\sigma\varphi\left(\frac{\mu}{\sigma}\right)+\mu\left(2\Phi\left(\frac{\mu}{\sigma}\right)-1\right).
\end{equation}
In a similar way, we can show the following proposition that provides expressions for the generalized kernel scores corresponding to the SCRPS, the robust CRPS, and the robust SCRPS. 

\begin{Proposition}\label{prop:gauss} In the case of Gaussian distributions, one has 
	\begin{align*}
	\Ssker_{1}(\pN(\hat{\mu},\hat{\sigma}^2),y) =&  -\sqrt{\pi}\varphi\left(\hat{z}\right)
	-\frac{\sqrt{\pi}\hat{z}}{2} \left(2\Phi\left(\hat{z}\right)-1\right) 
	- \frac1{2}\log\left(\frac{2\hat{\sigma}}{\sqrt{\pi}}\right),\\
	\Sker_{1,c}(\pN(\hat{\mu},\hat{\sigma}^2),y) =& \frac1{2}E(0,\sqrt{2}\hat{\sigma},c) - E(\hat{\mu}-y,\hat{\sigma},c),\\
	\Ssker_{1,c}(\pN(\hat{\mu},\hat{\sigma}^2),y) =& -\frac{E(\hat{\mu}-y,\hat{\sigma},c)}{E(0,\sqrt{2}\hat{\sigma},c)} - \frac1{2}\log(E(0,\sqrt{2}\hat{\sigma},c)),
	\end{align*}
	where $\hat{z} = \frac{\hat{\mu}-y}{\hat{\sigma}}$ and 
	\begin{align*}
	E(\mu,\sigma,c) =&  -\mu+\sigma\left(2\varphi\left(\frac{\mu}{\sigma}\right)-\varphi\left(\frac{c-\mu}{\sigma}\right)-\varphi\left(\frac{c+\mu}{\sigma}\right)\right) \\
	&+ (c-\mu)\Phi\left(\frac{\mu-c}{\sigma}\right)  + 2\mu\Phi\left(\frac{\mu}{\sigma}\right) + (\mu+c)\Phi\left(\frac{-c-\mu}{\sigma}\right).
	\end{align*}
\end{Proposition}

\begin{proof}
	Using the definition of $\Ssker_{1}$ in \eqref{eq:scaled_score}, with $g(x,y) = |x-y|$, together with \eqref{eq:fnorm} directly gives the desired formula for $\Ssker_{1}$.
	To get the result for the robust scores, note that if $X\sim\pN(\mu,\sigma^2)$, then 
	\begin{equation}\label{eq:normint}
	\begin{split}
	\int_a^b x\pi_X(x) dx =& \sigma\left(\varphi\left(\frac{a-\mu}{\sigma}\right)-\varphi\left(\frac{b-\mu}{\sigma}\right)\right) 
	+ \mu\left(\Phi\left(\frac{b-\mu}{\sigma}\right)-\Phi\left(\frac{a-\mu}{\sigma}\right)\right).
\end{split}
	\end{equation}
	Further, with $g_c(x) = 1(|x|<c)|x| + c1(|x|\geq c)$ we have 
	\begin{align*}
	E(\mu,\sigma,c) := \pE[g_c(X)] &= c(\pP(X\leq -c) + \pP(X\geq c)) + \int_{-c}^c|x|\pi_X(x)dx
	 := cA + B.
	\end{align*}
	Here
	$A = \Phi\left(\frac{-c-\mu}{\sigma}\right)+\Phi\left(\frac{\mu-c}{\sigma}\right)$,
	and using \eqref{eq:normint} we get
	\begin{align*}
	B =  \int_{0}^c x\pi_X(x)dx - \int_{-c}^0 x\pi_X(x)dx =&  \sigma\left(2\varphi\left(\frac{-\mu}{\sigma}\right)-\varphi\left(\frac{c-\mu}{\sigma}\right)-\varphi\left(\frac{-c-\mu}{\sigma}\right)\right) \\
	&+ \mu\left(\Phi\left(\frac{c-\mu}{\sigma}\right)+\Phi\left(\frac{-c-\mu}{\sigma}\right)-2\Phi\left(\frac{-\mu}{\sigma}\right)\right).
	\end{align*}
	Thus, 
	\begin{align*}
	cA + B =& \mu+\sigma\left(2\varphi\left(\frac{-\mu}{\sigma}\right)-\varphi\left(\frac{c-\mu}{\sigma}\right)-\varphi\left(\frac{-c-\mu}{\sigma}\right)\right) \\
	&+ (c-\mu)\Phi\left(\frac{\mu-c}{\sigma}\right) 
	- 2\mu\Phi\left(\frac{-\mu}{\sigma}\right) + (\mu+c)\Phi\left(\frac{-c-\mu}{\sigma}\right).
	\end{align*}
	Now, with $X\sim\pN(0,2\hat{\sigma}^2)$ and $X_y\sim \pN(\hat{\mu}-y,\hat{\sigma}^2)$, we have 
	\begin{align*}
	\Sker_{1,c}(\pN(\hat{\mu},\hat{\sigma}^2),y) &=\frac1{2}\pE[g_c(X_y)] - \frac1{2}\pE[g_c(X)]   
	= \frac1{2}E(0,\sqrt{2}\hat{\sigma},c)-E(\hat{\mu}-y,\hat{\sigma},c),\\
	\Ssker_{1,c}(\pN(\hat{\mu},\hat{\sigma}^2),y) &= -\frac{\pE[g_c(X_y)]}{\pE[g_c(X)]} - \frac1{2}\log(\pE[g_c(X)]) \\
	&= 
	-\frac{E(\hat{\mu}-y,\hat{\sigma},c)}{E(0,\sqrt{2}\hat{\sigma},c)} - \frac1{2}\log(E(0,\sqrt{2}\hat{\sigma},c)).
	\end{align*}
\end{proof}

In Example~\ref{ex1}, we showed the expected value of the CRPS and SCRPS when $Y$ followed a Gaussian distribution. Those values were computed using the following proposition.

\begin{Proposition}\label{prop:Egauss}
	Let $\mu_d = \mu - \hat{\mu}$ and $\sigma_d^2 = \hat{\sigma}^2+\sigma^2$, then
	\begin{align*}
	\Sker_{1}(\pN(\hat{\mu},\hat{\sigma}^2),\pN(\mu,\sigma^2))
	=& \frac{\hat{\sigma}}{\sqrt{\pi}} -  2\sigma_d\varphi\left(\frac{\mu_d}{\sigma_d}\right)
	+\mu_d\left(1- 2
	\Phi\left(\frac{\mu_d}{\sigma_d}\right)\right),\\
	\Ssker_{1}(\pN(\hat{\mu},\hat{\sigma}^2),\pN(\mu,\sigma^2)) =&
	-\frac{\sqrt{\pi}}{\hat{\sigma}}\left(\sigma_d\varphi\left(\frac{\mu_d}{\sigma_d}\right) - \frac{\mu_d}{2}  + \mu_d\Phi\left(\frac{\mu_d}{\sigma_d}\right) \right)
	- \frac1{2}\log\left(\frac{2\hat{\sigma}}{\sqrt{\pi}}\right),\\
	\Sker_{1,c}(\pN(\hat{\mu},\hat{\sigma}^2),\pN(\mu,\sigma^2)) =&
	\frac1{2}E(0,\sqrt{2}\hat{\sigma},c) - E(\mu_d,\sigma_d,c),\\
	\Ssker_{1,c}(\pN(\hat{\mu},\hat{\sigma}^2),\pN(\mu,\sigma^2)) =&
	-\frac{E(\mu_d,\sigma_d,c)}{E(0,\sqrt{2}\hat{\sigma},c)} - \frac1{2}\log(E(0,\sqrt{2}\hat{\sigma},c)).
	\end{align*}
	
\end{Proposition}

\begin{proof}
	Note that if $X\sim\pN(\mu,\sigma)$, then
	$$
	\pE[\varphi(X)] = \frac1{\sqrt{\sigma^2+1}}\varphi\left(\frac{\mu}{\sqrt{\sigma^2+1}}\right), \quad 
	\mbox{and}\quad\pE[\Phi(X)] = \Phi\left(\frac{\mu}{\sqrt{1+\sigma^2}}\right).
	$$
	It is also easy to show that $\pE[X\Phi(X)] = \mu\pE[\Phi(X)] + \sigma^2\pE[\varphi(X)]$. 
	Define $\tilde{\mu} = (\mu - \hat{\mu})/\hat{\sigma}$ and $\tilde{\sigma} = \sigma/\hat{\sigma}$, and let $\tilde{Y}\sim\pN(\tilde{\mu},\tilde{\sigma}^2)$, then
	$$
	\pE[\varphi(\tilde{Y})] = \frac{\hat{\sigma}}{\sigma_d}\varphi\left(\frac{\mu_d}{\sigma_d}\right), \quad \pE[\Phi(\tilde{Y})] = \Phi\left(\frac{\mu_d}{\sigma_d}\right).
	$$
	Thus, for the CRPS we have
	\begin{align*}
	\pE_{\pN(\mu,\sigma^2)}[\Sker_{1}(\pN(\hat{\mu},\hat{\sigma}^2),y) ] 
	&= 
	\hat{\sigma}\left(\frac1{\sqrt{\pi}} - 2\pE[\varphi(\tilde{Y})] - 2\pE[\tilde{Y}\Phi(\tilde{Y})] + \pE[\tilde{Y}]\right)\\
	&= 
	\hat{\sigma}\left(\frac1{\sqrt{\pi}} - 2(1+\tilde{\sigma}^2)\pE(\varphi(\tilde{Y})) - 2\tilde{\mu}\pE[\Phi(\tilde{Y})]  +\tilde{\mu}\right)\\
	&=
	\frac{\hat{\sigma}}{\sqrt{\pi}} - 2\sigma_d\varphi\left(\frac{\mu_d}{\sigma_d}\right) - 2\mu_d\Phi\left(\frac{\mu_d}{\sigma_d}\right)  +\mu_d.
	\end{align*}
	For the scaled CRPS, similar calculations give
	\begin{align*}
	\pE_{\pN(\mu,\sigma^2)}[\Ssker_{1}(\pN(\hat{\mu},\hat{\sigma}^2),y)] 
	&= 
	-\sqrt{\pi}\pE[\varphi(-\tilde{Y})] +\frac{\sqrt{\pi}}{2}(2\pE[\tilde{Y}\Phi(-\tilde{Y})]+\pE[\tilde{Y}])
	-  \frac1{2}\log\left(\frac{2\hat{\sigma}}{\sqrt{\pi}}\right)\\
	&= 
	-\sqrt{\pi}\pE[\varphi(\tilde{Y})] +\frac{\sqrt{\pi}}{2}\pE(\tilde{Y}) - \sqrt{\pi}\pE[\tilde{Y}\Phi(\tilde{Y})]
	-  \frac1{2}\log\left(\frac{2\hat{\sigma}}{\sqrt{\pi}}\right)\\
	&= 
	-\frac{\sqrt{\pi}}{\hat{\sigma}}\left(\sigma_d\varphi\left(\frac{\mu_d}{\sigma_d}\right) - \frac{\mu_d}{2}  + \mu_d\Phi\left(\frac{\mu_d}{\sigma_d}\right) \right)
	-  \frac1{2}\log\left(\frac{2\hat{\sigma}}{\sqrt{\pi}}\right).
	\end{align*}
	Finally, 
	to obtain the desired expressions for the robust scores, we just have to verify that 
	$
 \pE_{\pN(\mu,\sigma^2)}[E(\hat{\mu}-y,\hat{\sigma},c)] =  E(\mu_d,\sigma_d,c).
	$ 
	To that end, we define $\tilde{Y}_1\sim \pN(-(\mu_d-c)/\hat{\sigma},\tilde{\sigma}^2)$, $\tilde{Y}_2 \sim \pN(\mu_d/\hat{\sigma},\tilde{\sigma}^2)$, and $\tilde{Y}_3 \sim \pN((\mu_d-c)/\hat{\sigma},\tilde{\sigma}^2)$, and obtain
	\begin{align*}
	\pE_{\pN(\mu,\sigma^2)}[E(\hat{\mu}-y,\hat{\sigma},c)] =& -\mu_d 
	+ 2\hat{\sigma}(\pE[\varphi(\tilde{Y}_2])+ \pE[\tilde{Y}_2\Phi(\tilde{Y}_2)]) \\
	&\quad - \hat{\sigma}(\pE[\varphi(\tilde{Y}_1)] + \pE[\tilde{Y}_1\Phi(\tilde{Y}_1)]).
	\end{align*}
	Evaluating the expectations and simplifying gives the desired result.
\end{proof}

\section{Generalized entropy characterizations}\label{sec:entropy}
In this section, we want to highlight the importance of the function $H(\P)=S(\P,\P)$ and how it can be used to understand the behavior of the scoring rule $S$ in scenarios with varying uncertainty. The function $H(\P)$ can be seen as a measure of the variability of the probability measure $\P$, and is often referred to either as the uncertainty function or as the generalized entropy \citep{Dawid1998,gneiting2007strictly}.

It should be noted that $H(\P)$ only depends on the predictive model $\P$ and not on the observed data. Thus by choosing to use a certain scoring rule $S$, an implicit ordering of the possible measures is made through $H$ prior to observing any data. Hence, it is important to be aware that a choice is made and to understand how this affects the model selection. 

For the kernel scoring rules we have $H(\P) = 0.5\E_{\P,\P}\left[ g(X,Y)\right])$, whereas the generalized proper kernel scoring rules have $H(\P) = h(\E_{\P,\P}\left[ g(X,Y)\right])$. An advantage with the generalized proper kernel scoring rules is therefore that they, for each kernel $g$, provide a family of proper scoring rules with a wide range of generalized entropies determined by $h$. 
Let us now examine the generalized entropy for a set of measures with variable scale.
\begin{example}
	Consider a family of probability measures $\{\P_\sigma\}$ that differ only through their scaling parameter $\sigma$.
	For the kernel $g(x,y)=|x-y|^{\alpha}$ with $\alpha \in (0,2]$, the kernel scoring rule then has generalized entropy 
	$$
	H^{ker}(\P_{\sigma}) = \sigma^{\alpha} \E_{\P_{1},\P_{1}} \left[|X-Y|^{\alpha}\right].
	$$
	The corresponding standarized kernel scoring rule has generalized entropy
	$$
	H^{sta}(\P_{\sigma}) = \alpha \log(\sigma) + \log \left( \E_{\P_{1},\P_{1}} \left[|X-Y|^{\alpha}\right] \right).
	$$
	In a spatial setting like the one studied in Section~\ref{sec:spatstat}, $H^{ker}(\P_{\sigma})$ compared to $H^{sta}(\P_{\sigma})$ will therefore be more sensitive how the variance is chosen for observations at locations far from other locations, which typically have large values of $\sigma$, while being less sensitive for observations at locations close to other locations, which typically have small values of $\sigma$. Whether this is a desirable feature is something that needs to be decided by the person evaluating the forecasts when choosing which entropy function to use. 
\end{example}

Recently, it has been suggested to study the distribution of scoring rules for data sets rather than only considering the mean score \citep[see, e.g.,][]{naveau2018forecast,taillardat2019extreme}. That is, to consider the distribution of $S(\P_{\hat{\theta}_i},Y_i)$ when $\P_{\hat{\theta}_i}$ is the measure which predicts observation $Y_i$. When exploring the distribution we argue that it often makes sense to study the distribution of $H(\P_{\hat{\theta}_i})$ and $S(\P_{\hat{\theta}_i}, Y_i)- H(\P_{\hat{\theta}_i})$ separately. The distribution of the first term, $H(\P_{\hat{\theta}_i})$, provides no information about the fit of the model to the data but instead provides information about how much variability one expects the data to have a priori. The second term, $S(\P_{\hat{\theta}_i}, Y_i)- H(\P_{\hat{\theta}_i})$, on the other hand gives an indication about how close the model fits the data, where a zero value indicates a perfect calibration in the sense of the generalized entropy. 

Thinking of the two terms in a regression-analogy, the first term explores the variability of the covariates given the model in the defined entropy sense and it is importantly independent of the data. The second term, that uses the data, explores the difference between the observed score and the expected score if the predictive model was the true distribution.

It is important to keep these two terms in mind when exploring distributions of scores over predictions. Especially troublesome is to fail to notice that the variability of the first term, which often is substantial, is data independent. Ignoring this when exploring the distributions for evaluating a model fit to data may lead to incorrect conclusions.

\begin{figure}[!t]
	\begin{center}
		\begin{minipage}{0.25\linewidth}
			\begin{center}
				CRPS\\
				\includegraphics[width=\linewidth]{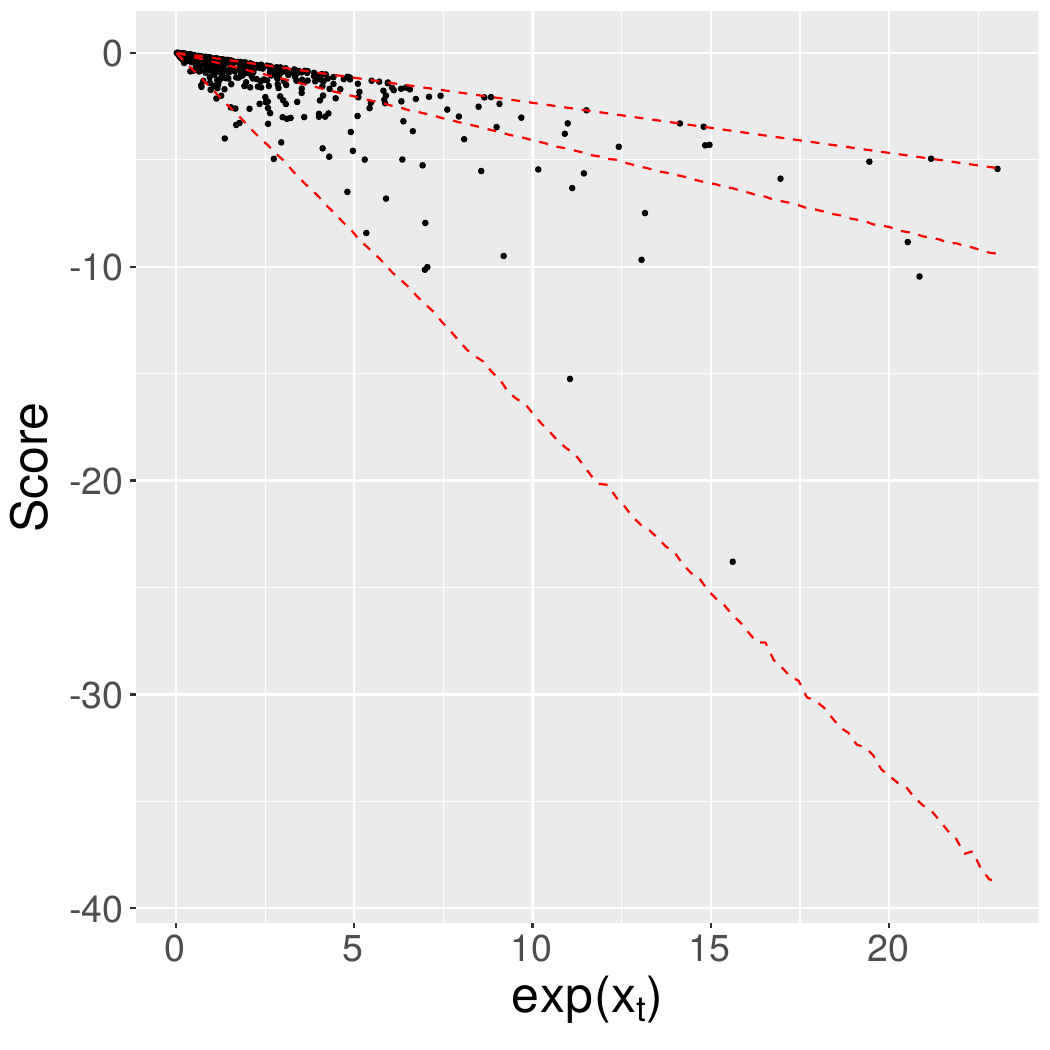}
				\includegraphics[width=\linewidth]{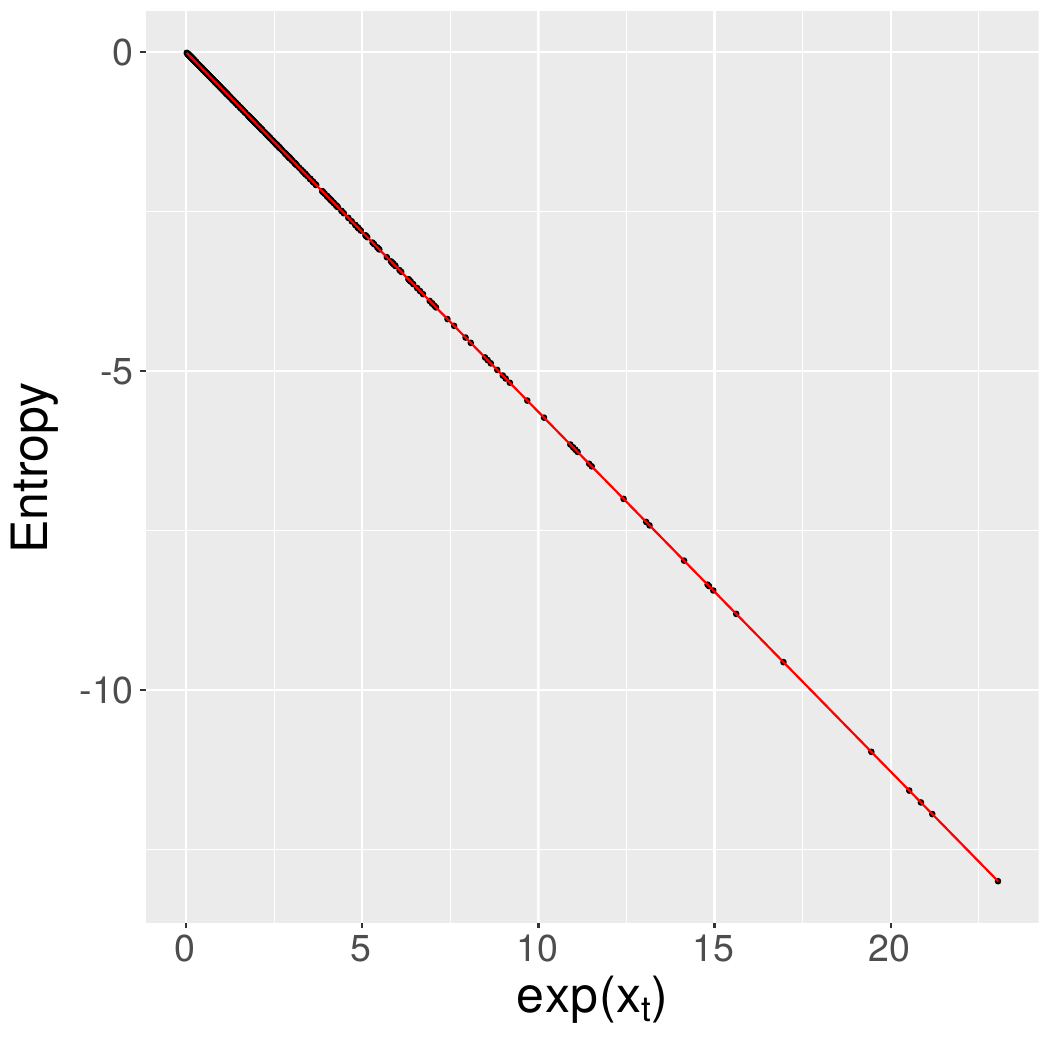}
				\includegraphics[width=\linewidth]{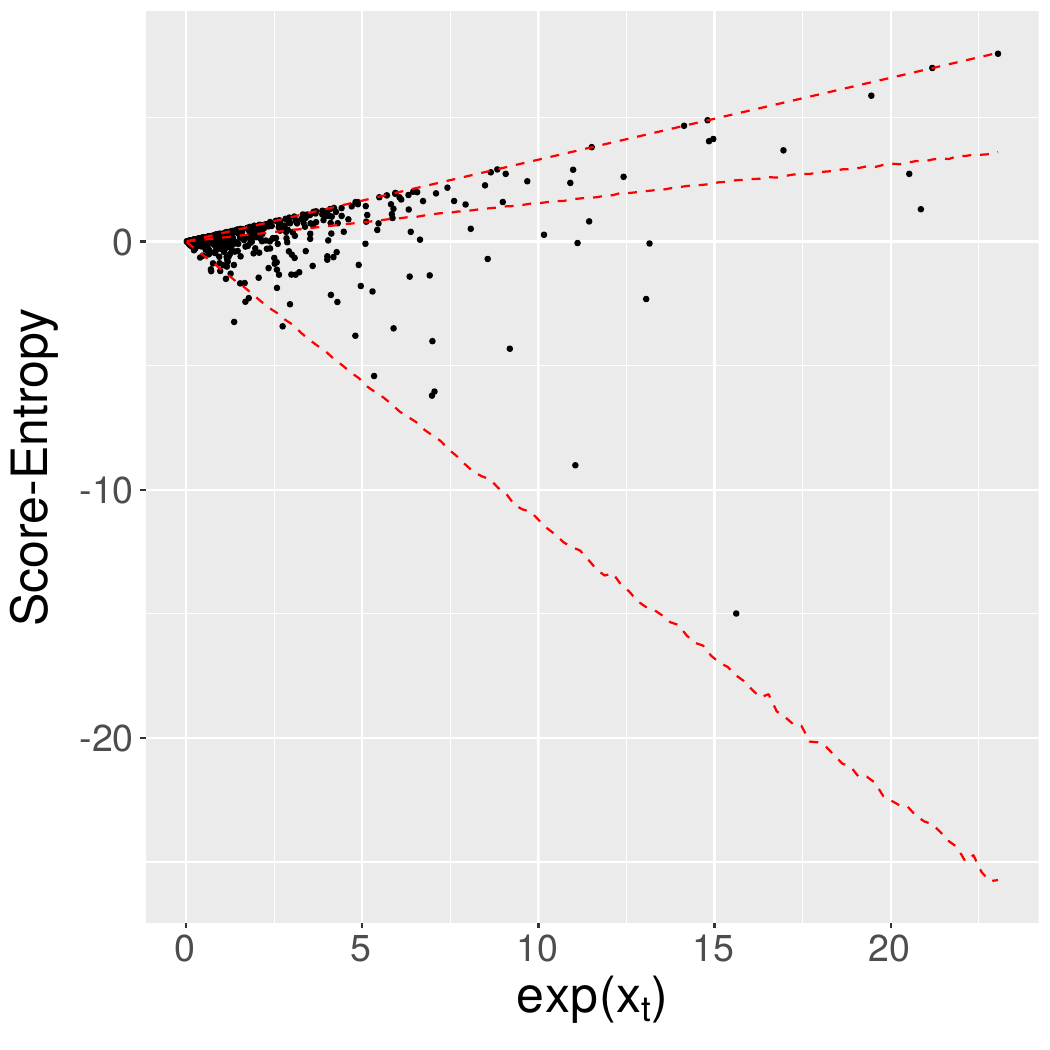}\\			
			\end{center}
		\end{minipage}
		\begin{minipage}{0.25\linewidth}
			\begin{center}
				SCRPS\\
				\includegraphics[width=\linewidth]{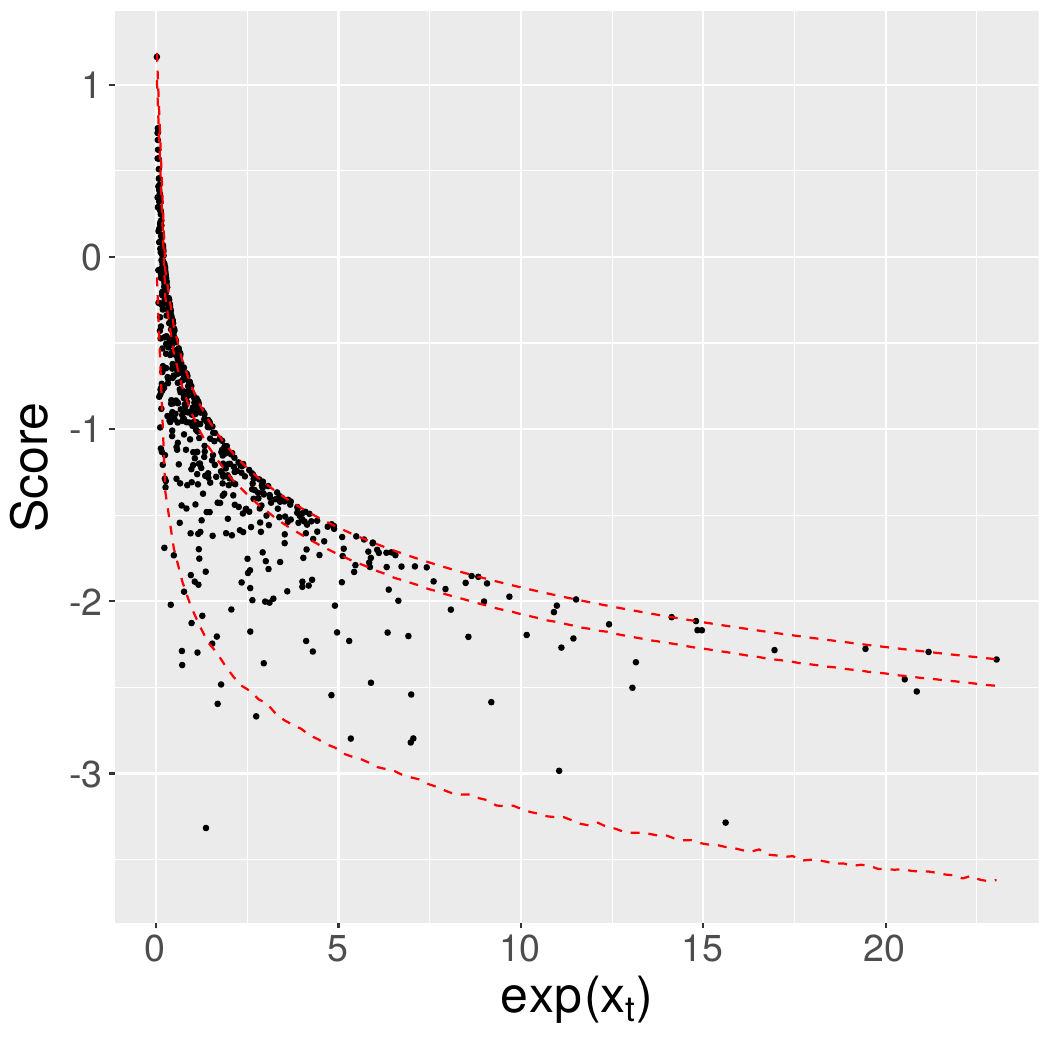}
				\includegraphics[width=\linewidth]{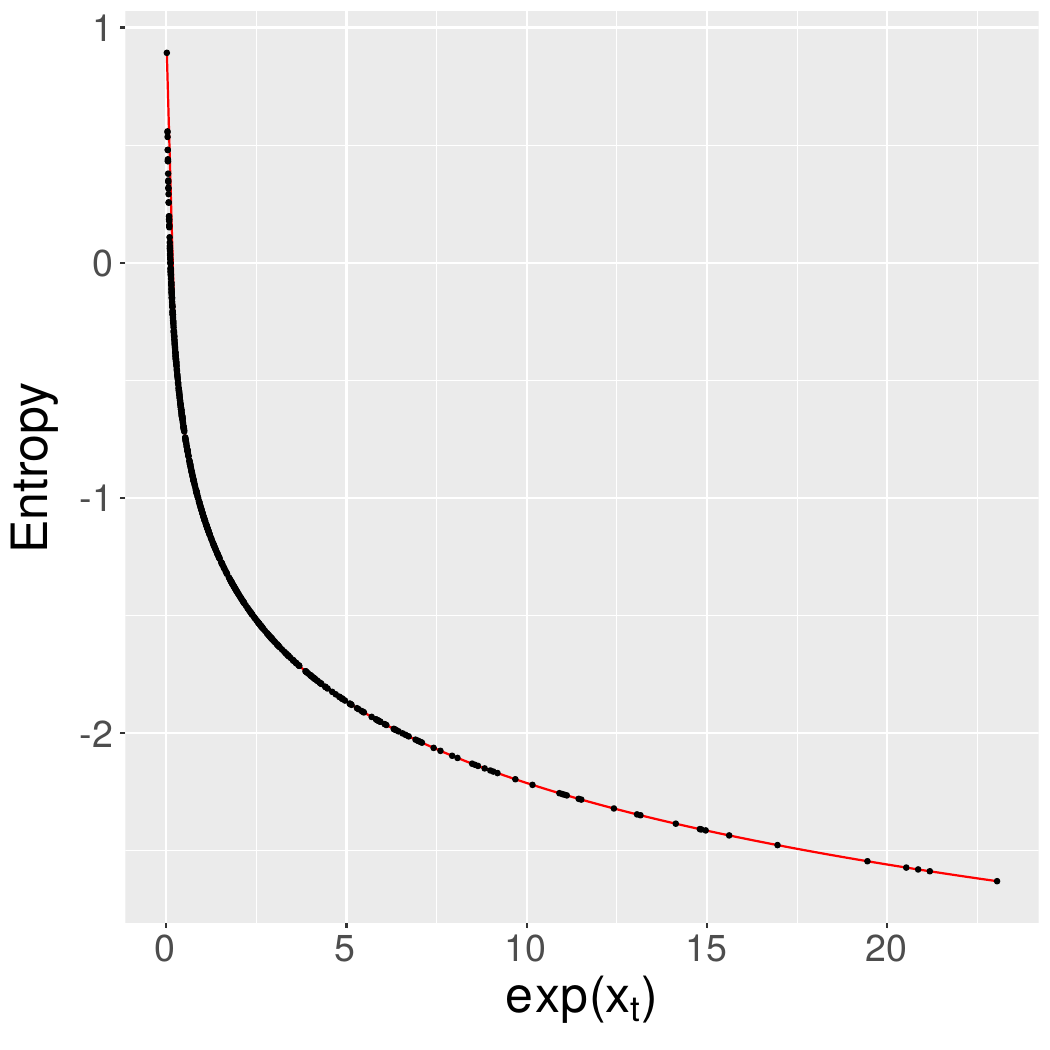}
				\includegraphics[width=\linewidth]{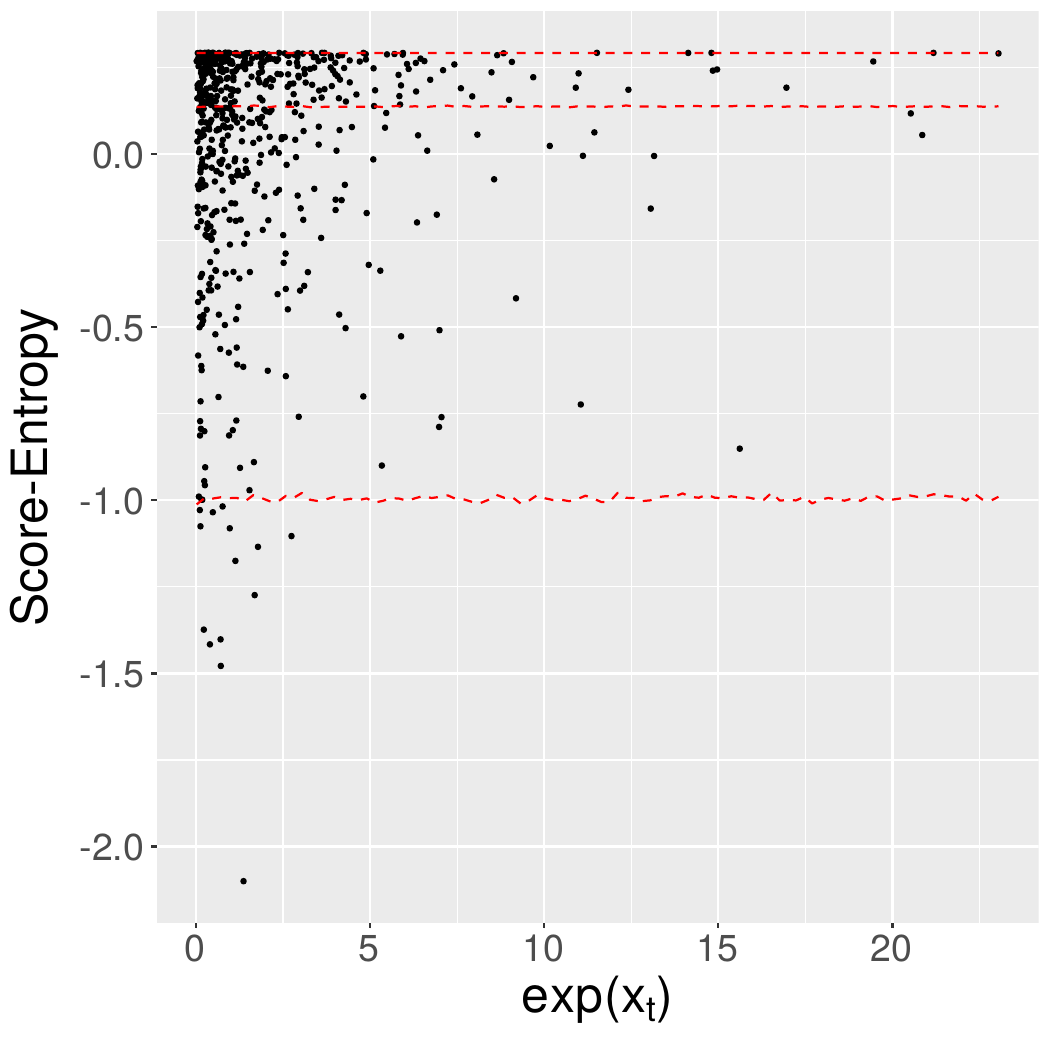}\\
			\end{center}
		\end{minipage}
		\begin{minipage}{0.25\linewidth}
			\begin{center}
				log score\\
				\includegraphics[width=\linewidth]{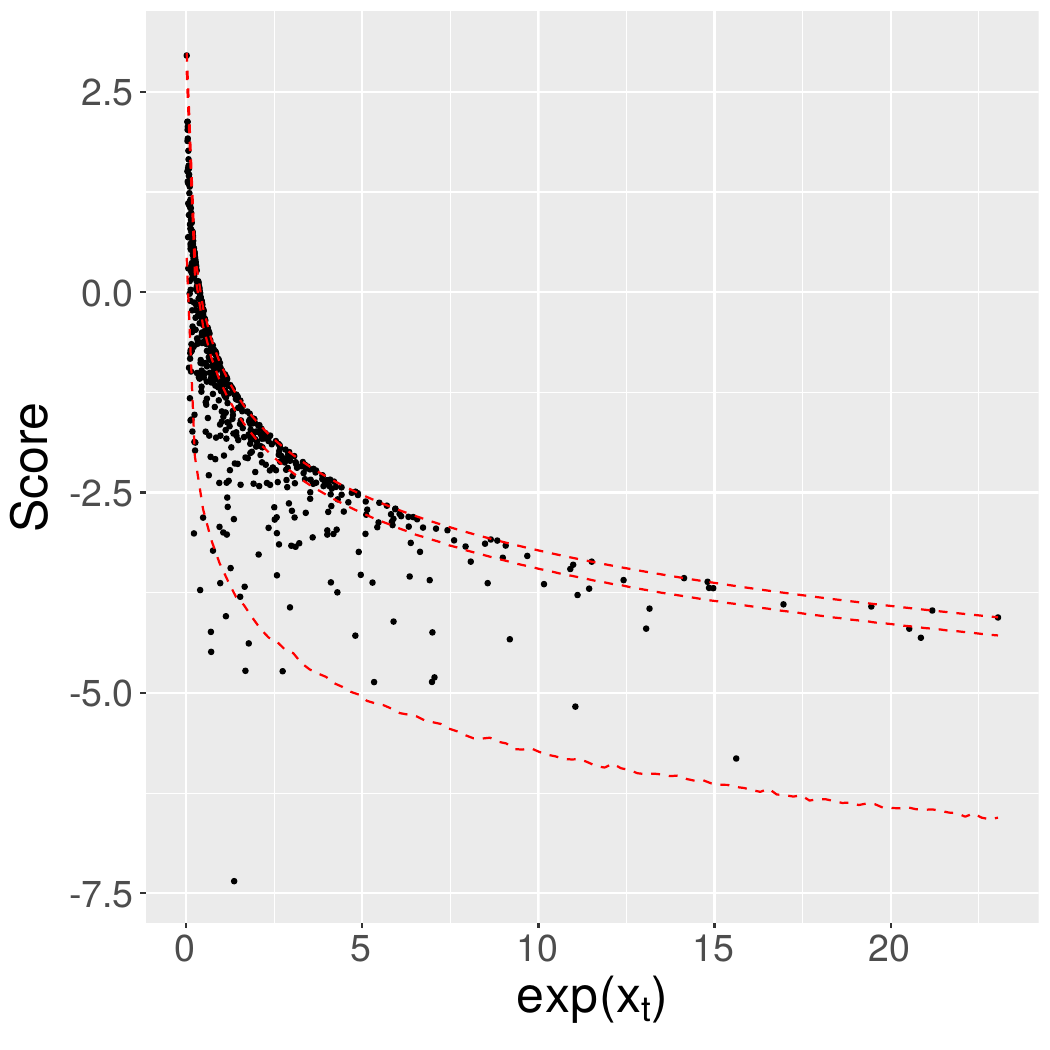}
				\includegraphics[width=\linewidth]{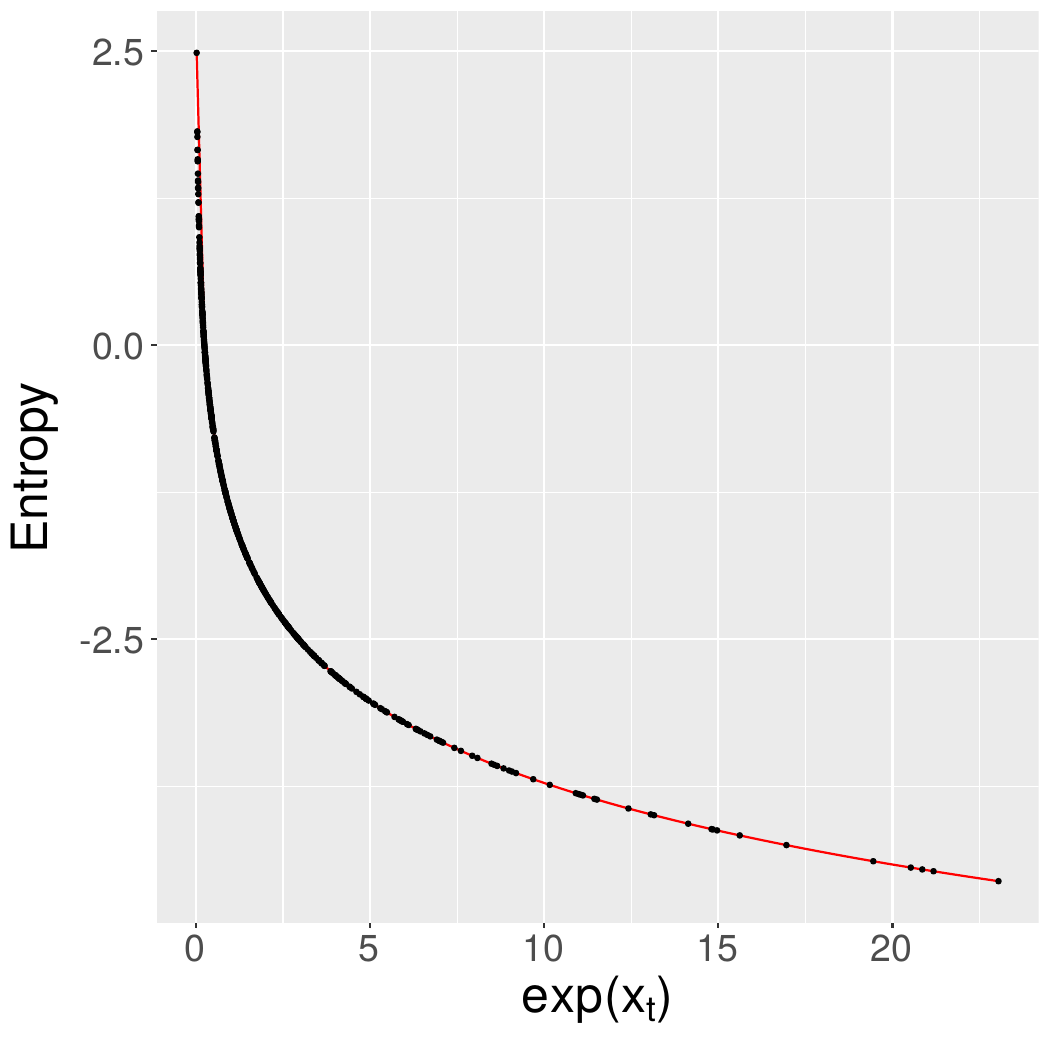}
				\includegraphics[width=\linewidth]{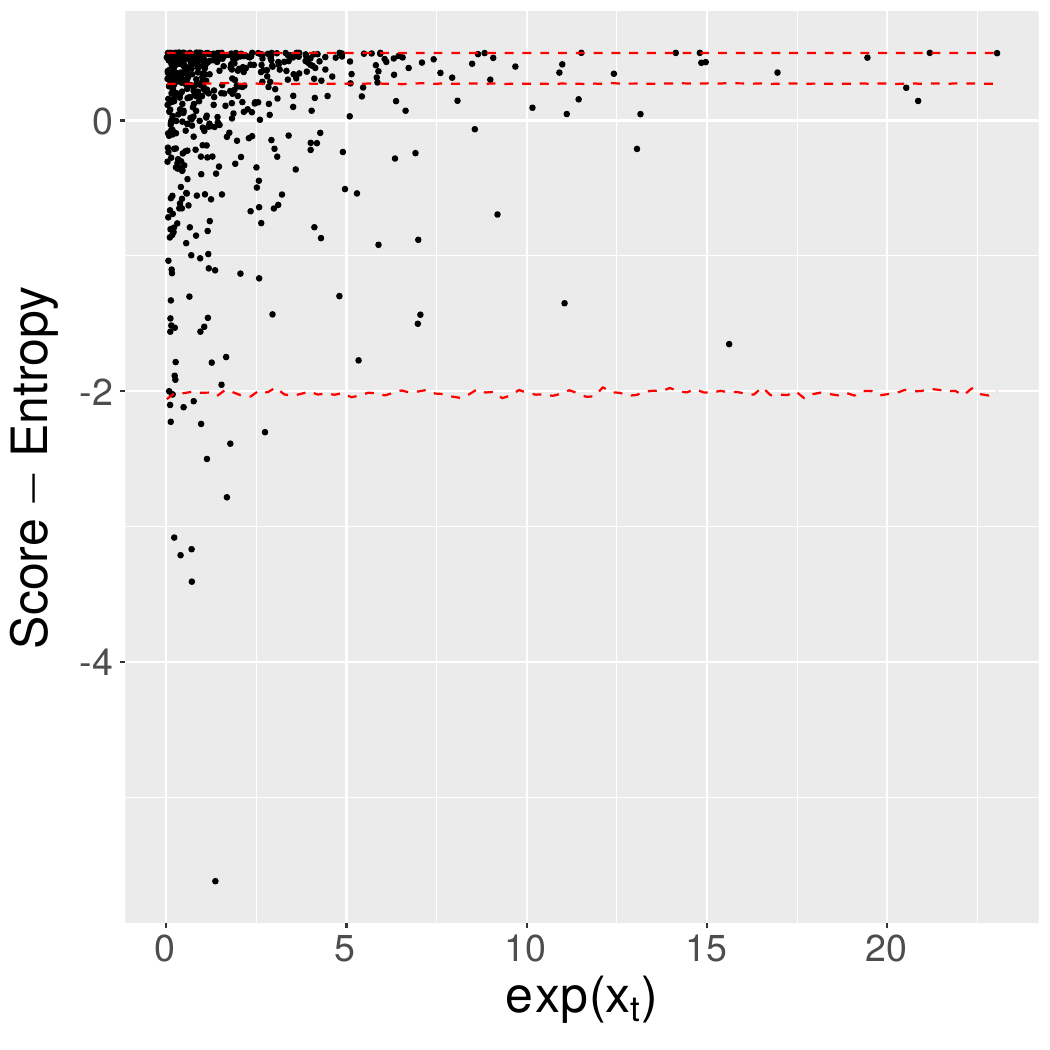}
			\end{center}
		\end{minipage}		
	\end{center}
	\caption{Three scoring rules for a stochastic volatility model. The black dots are the values for each observation.  The three rows show the scoring rule (top), the corresponding entropy with theoretical values as solid lines (middle), and score minus the entropy (bottom). The dashed red lines are estimated $2.5\%,50\%$ and $97.52.5\%$ quantiles. 
	}
	\label{fig:scoringRuleVol}
\end{figure}

As an example of the different terms, Figure \ref{fig:scoringRuleVol} shows the CRPS, the scaled CRPS, and the log-score for the observations in Figure  \ref{fig:SimVol}.  The middle row shows the entropy of the scoring rules, where one clearly sees the linear cost of increased variance for the CRPS and the logarithmic cost of increased variance for the SCRPS and the log-score.  The bottom row displays the difference between the score and the entropy. Here it is interesting to note that the variability of the term increases as a function of the volatility for CRPS, whereas the standardized score has the same distribution regardless of the variance. Of general interest is also that it is hard to see a difference between the SCRPS and the log-score.

\section{Further proofs}\label{sec:proofs}

\begin{proof}[Proof of Proposition~\ref{prop:scale_ls}]
	Let $S(\Q_\theta,y) = \log(q_\theta(Y))$ denote the log-score, where $q_\theta$ is the density of $\Q$, and recall that $S(\Q_\theta,\Q) = \E_\Q[\log(q_\theta(Y))]$.
	Using that $S$ is a proper scoring rule, we have that $\nabla_{\theta} S \left( \Q_{\theta  },\Q\right)|_{\Q = \Q_\theta} = {\bf 0}$.  Thus, by Taylor expansion 
	we get
	\begin{equation}\label{eq:taylor}
	S \left( \Q_{\theta+ t\sigma r},\Q_{\theta}\right) - S \left( \Q_{\theta},\Q_{\theta}\right) = t^2\sigma^2  r^T\nabla^2_{\theta} S \left( \Q_{\theta },\Q_{\theta}\right)  r + o(t^2).
	\end{equation}
	Here $s(\theta)=\nabla^2_{\theta} S \left( \Q_{\theta },\Q\right)|_{\Q = \Q_\theta}$ exists and is continuous by Assumption \ref{ass1}, and from classical results in likelihood theory \cite[see][]{lehmann1983theory} we get
	$$
	s(\theta) = \left. \nabla^2_{\theta} \E_{\Q}[\log q_{\theta}(X)] \right|_{\Q=\Q_{\theta}} =  \frac{1}{\sigma^2}H_{\Q},
	$$
	for some $2 \times 2$ matrix $H_{\Q}$ independent of $\theta$. Thus, 
	$
	s(\Q_\theta) = s(\theta).
	$
\end{proof}


\begin{Lemma}
	\label{lem1} Let $\alpha \in (0,2]$ and $c\in[0,\infty]$ and define $g_c(x,y)=\mathbb{I}(|x-y|<c)|x-y|^{\alpha}$. Let $\Q$ and $\P$ be probability measures  and $\E_{\P}[|X|^\alpha] <\infty$. Then if Assumption \ref{ass1} holds for $\P$, 
	\begin{align}
	\label{eq:grad1}
	\nabla_{\theta} \E_{\P_{\theta},\Q}[g_c(X,Y)] &= -
	\sigma^{-1}\E_{\P,\Q}\left[
	g_c(\sigma X + \mu,Y)  
	v(X)\right]), \\
	\label{eq:Hessian1}
	\nabla^2_{\theta} \E_{\P_{\theta},\Q}[g_c(X,Y)] =& 
	\sigma^{-2}\E_{\P,\Q}\left[
	g_c(\sigma X + \mu,Y) H(X) \right],
	\end{align}
	for $\sigma>0$ where $v(X) = (\Psi'(X), \Psi'(X)  X + 1)^T$ and
	\begin{align*}
	H(X) =&v(X)v(X)^T 
	+
	\Psi''(X)
	\begin{bmatrix}
	 1& X   \\
	X  &  X^2
	\end{bmatrix}
	+
	\begin{bmatrix}
	0 & \Psi'(X)   \\
	\Psi'(X)  &  2\Psi'(X) X +1
	\end{bmatrix} .
	\end{align*}
	Further, if we apply the gradient to both arguments of $\E_{\P_{\theta},\P_{\theta}}[g_c(X,Y)]$, we get 
		\begin{align}
		\label{eq:GradCross}
			\sigma^{\alpha-2}\E_{\P,\P}\left[
		g( X ,Y) \mathbb{I}\left(|X-Y|<\frac{c}{\sigma}\right)
		v(X)
		v(Y)^T
\right].
		\end{align}
\end{Lemma}

\begin{proof}
To show \eqref{eq:grad1}, we start by considering the derivative with respect to $\sigma$.
Using the mean value theorem, we have
\begin{align*}
\frac{\pd}{\pd\sigma} \E_{\P_{\theta},\Q}[g_c(X,Y)] 
&= \lim_{h\rightarrow 0}
	\frac{1}{h} \iint  g_c(x,y)\left(\frac{1}{\sigma + h} p\left(\frac{x-\mu}{\sigma+h}\right) -\frac{1}{\sigma} p\left(\frac{x-\mu}{\sigma}\right)\right)dx \Q(dy)  \\
	&= \lim_{h\rightarrow 0} \iint  g_c(x,y) \frac{d}{d\sigma}\left.\frac{1}{\sigma} p\left(\frac{x-\mu}{\sigma}\right)\right|_{\sigma=\sigma+h^*} dx \Q(dy)
\end{align*}
for some $h^* \in (0,h)$. Evaluating the derivative and using the variable transformation $\tilde x = \frac{x-\mu}{\sigma}$ we get that  this expression equals
\begin{align}\label{eq:lemma1proof}
\lim_{h^* \rightarrow 0}
- \int \int  g_c( (\sigma + h^*) \tilde x + \mu,y)  \frac{1}{\sigma + h^*} \left(\Psi'(\tilde x) \tilde x + 1\right) p(\tilde x) d \tilde x \Q(dy).
\end{align}
Now since $g_c(x,y)\leq g(x,y)$ which is a negative definite kernel it follows  by the same argument as that used in the proof of Theorem \ref{the:kernelrob} that
$$
 g_c( (\sigma + h^*) \tilde x + \mu,y)  \leq C \left( (\sigma + h^*)^{\alpha} |\tilde x|^{\alpha} + |\mu|^{\alpha} +  |y|^{\alpha} \right)
$$
for some $C>0$. Combining this bound with Assumption \ref{ass1} shows that the integral in \eqref{eq:lemma1proof}  for each $h^*$ can be bounded by an integrable function. Thus, using the dominated convergence theorem we can move the limit into the integral, which gives that $$\frac{\pd}{\pd\sigma} \E_{\P_{\theta},\Q}[g_c(X,Y)] = -
\sigma^{-1}\E_{\P,\Q}\left[ g_c(\sigma X + \mu,Y) (\Psi'(X)  X + 1)\right].$$ 
The corresponding expression for the derivative with respect to $\mu$ can be shown by the same reasoning. Also the expressions in \eqref{eq:Hessian1} and \eqref{eq:GradCross} are shown in the same way, by differentiating twice and using the argument above. For brevity, we omit these calculations. 
\end{proof}

\begin{proof}[Proof of Propositions \ref{prop:kernel_scale} and  \ref{prop:scale_rob}]
As in the proof of Proposition \ref{prop:scale_ls}, Taylor expansion and using that $S$ is a proper scoring rule gives that $S$ satisfies \eqref{eq:taylor} for both propositions. 
We apply Lemma \ref{lem1} with $\mathbb{Q}= \mathbb{P}_{\theta}=\mathbb{Q}_{\theta}$, then \eqref{eq:Hessian1} gives 
 \begin{align*}	
\nabla^2_{\theta} S \left( \Q_{\theta },\Q_{\theta}\right)&=  \nabla^2_{\theta} \E_{\Q_{\theta},\Q_{\theta}}[g_c(X,Y)] =
\sigma^{-2}\E_{\Q,\Q}\left[
g(X,Y) H(X) \mathbb{I}\left(|X-Y|<\frac{c}{\sigma}\right) \right],
\end{align*}
which shows Proposition \ref{prop:scale_rob}. Setting $c=\infty$ in this expression gives 
 $$
\nabla^2_{\theta} S \left( \Q_{\theta },\Q_{\theta}\right)=  \nabla^2_{\theta} \E_{\Q_{\theta},\Q_{\theta}}[g_\infty(X,Y)] =
 \sigma^{\alpha-2}\E_{\Q,\Q}\left[
 g(X,Y) H(X) \right],
 $$
which shows Proposition \ref{prop:kernel_scale}.
\end{proof}

\begin{Lemma}
	\label{lemma:sks}
	Let $\P$ be a probability measure satisfying Assumption \ref{ass1}  for $\alpha\in (0,2]$, with $\E_{\P}[|X|^\alpha] <\infty$. If $g(x,y)=|x-y|^\alpha$ then Hessian with respect to the first argument of the standardized kernel scoring rule at $\Q=\P_{\theta}$ is
	$$
	\left.\nabla^2_{\theta}	S^{-\frac{1}{2} \log(x)}_g(\P_{\theta}, \Q)\right|_{\Q= \P_\theta} = \sigma^{-2} H_{\P}
	$$
	where $H_{\P}$ is a $2 \times 2$ matrix independent of $\mu$ and $\sigma$. 
\end{Lemma}
\begin{proof}
To simplity notation, let 
$\E_{P,Q} = \E_{\P_{\theta},\Q}[g(X,Y)]$,
	$\mv{\E}_{\dot P,Q} = 	\nabla_{\theta} \E_{\P_{\theta},\Q}[g(X,Y)]$, and
	$\mv{\E}_{Q,\dot P} = 	\nabla_{\theta} \E_{\Q,\P_{\theta}}[g(X,Y)]$.
	Now, straightforward calculations give that
	\begin{align*}
	\nabla_{\theta} \frac{1}{2}\log(\E_{P,P}) &= \frac{1}{\E_{P,P}} \mv{E}_{\dot P,P},\\
	\nabla^2_{\theta} \frac{1}{2}\log(\E_{P,P}) &= -\frac{2}{\E^2_{P,P}} \mv{E}_{\dot P,P} \mv{E}_{\dot P,P}^T
	+ \frac{1}{\E_{P,P}} \mv{E}_{\dot P, \dot P}
	+  \frac{1}{\E_{P,P}} \mv{E}_{\ddot P,P},
	\end{align*}
and
\begin{align*}
\nabla_{\theta}\frac{\E_{P,Q}}{\E_{P,P}} =&
\frac{1}{\E_{P,P}}  \mv{\E}_{\dot P,Q} - \frac{2\E_{P,Q}}{\E^2_{P,P}} \mv{\E}_{\dot P,P} , \\
	\nabla^2_{\theta} \frac{\E_{P,Q}}{\E_{P,P}} =& 
	-\frac{2}{\E^2_{P,P}}  \mv{\E}_{ \dot P,Q} \mv{E}^T_{\dot P,P}  -\frac{2}{\E^2_{P,P}} \mv{\E}_{P,Q}\mv{\E}_{\dot P,P}^T  - \frac{2\E_{P,Q}}{\E^2_{P,P}} \mv{\E}_{\dot P, \dot P} 
	\\
	&+\frac{2^3\E_{P,Q}}{\E^3_{P,P}} \mv{\E}_{\dot P,P}\mv{\E}_{\dot P,P}^T  +
	  \frac{1}{\E_{P,P}}  \mv{\E}_{ \ddot P,Q} - \frac{2\E_{P,Q}}{\E^2_{P,P}} \mv{\E}_{\ddot P,P}.
\end{align*}
Evaluating the last term at $Q=P$ gives
\begin{align*}
\left.\nabla^2_{\theta} \frac{\E_{P,Q}}{\E_{P,P}}\right|_{Q=P} =& 
-\frac{2}{\E_{P,P}}  \mv{\E}_{ \dot P, \dot P}  +\frac{4}{\E^2_{P,P}} \mv{\E}_{\dot P,P}\mv{\E}_{\dot P,P}^T  -\frac{1}{\E_{P,P}}  \mv{\E}_{ \ddot P,P}.
\end{align*}
Putting it all together yields
\begin{align*}
\left.\nabla^2_{\theta}\left(	S^{-\frac{1}{2} \log(x)}_g(\P_{\theta}, \Q) \right)\right|_{\Q = \P_{\theta}} =
\frac{1}{\E_{P,P}}  \mv{\E}_{ \dot P, \dot P}  -\frac{2}{\E^2_{P,P}} \mv{\E}_{\dot P,P}\mv{\E}_{\dot P,P}^T.  
\end{align*}
The result follows since  $\mv{\E}_{ \dot P,\dot P}$ is given by \eqref{eq:GradCross}  and $\mv{\E}_{\dot P,P}$ is given by \eqref{eq:grad1} with $c=\infty$.
\end{proof}

\begin{proof}[Proof of Proposition \ref{prop:scale_stand}]
	As in the proof of Proposition \ref{prop:scale_ls}, Taylor expansion and using that $S$ is a proper scoring rule gives that the score satisfied \eqref{eq:taylor}.
	The results now follows from Lemma \ref{lemma:sks}, since 
	$
	\nabla^2_{\theta}  S(\Q_{\theta},\Q)|_{\Q=\Q_{\theta}} = \frac{1}{\sigma^2}H_{\Q}
	$
	for some $2 \times 2$ matrix $H_{\Q}$ independent of $\theta$. 
\end{proof}

	\end{appendix}

\end{document}